\documentclass[utopia]{nmd/nmd-article}
\pdfoutput=1
\usepackage[letterpaper,headsep=10bp,headheight=15bp,left=1.7cm,right=1.7cm,top=2.2cm,bottom=2.5cm]{geometry}

\usepackage{amsfonts}
\usepackage{amsmath}
\usepackage{amsthm}
\usepackage{amssymb}
\usepackage{stackengine}
\usepackage{longtable}
\usepackage{array}
\usepackage{float}
\usepackage{wasysym}
\usepackage{enumitem}

\usepackage{titletoc}

\newlist{algorithmSteps}{enumerate}{5}
\setlist[algorithmSteps]{label={\arabic*.}, ref={\arabic*}, nosep}

\newcommand{\myTrig}{\mathcal{T}\!}
\newcommand{\postNorm}{\widehat{\phantom{x}}~}
\newcommand{\posLightCone}{\mathbb{L}_{+}^3}
\newcommand{\emphasisText}[1]{{\bf #1\/}}
\newcommand{\emphasisDef}[1]{{\bf #1\/}}
\newcommand{\from}{\colon}

\newcommand{\hypPlane}[1]{\hypPlaneOp(#1)}
\newcommand{\oxford}{,}

\DeclareMathOperator{\hypPlaneOp}{\perp}

\DeclareMathOperator{\ImPart}{Im}
\DeclareMathOperator{\Id}{Id}
\DeclareMathOperator{\myO}{O}
\DeclareMathOperator{\sgn}{sgn}
\DeclareMathOperator{\Eq}{Eq}
\DeclareMathOperator{\Tiles}{Tiles}

\floatstyle{ruled}
\newfloat{algorithm}{}{}
\floatname{algorithm}{Algorithm}

\makeatletter
\let\c@algorithm=\c@subsubsection
\makeatother

\title{Tiling Hyperbolic Manifolds:\\ Algorithms and Applications} 
\author{Matthias Goerner}
\email{enischte@gmail.com}

\begin{document}

\begin{abstract}
We introduce a new tiling algorithm for hyperbolic $3$-manifolds. We use it to compute the maximal cusp area matrix; this completely characterizes the space of all embedded and disjoint cusp neighborhoods. As another application of our work, we find the Epstein-Penner decomposition answering a challenge of Sakuma and Weeks.
We furthermore provide the refinements needed to make our algorithm verified: producing intervals provably containing the correct answer. As key ingredient for our work and perhaps of independent interest, we give new and simpler expressions for the distances between points, lines\oxford{} and planes in the hyperboloid model.
\end{abstract}

\maketitle

\dottedcontents{section}[3.2em]{\bfseries}{1.8em}{0.7pc}

\setcounter{tocdepth}{1}
\tableofcontents

\section{Introduction}

Hodgson and Weeks \cite{hwcensus} used a tiling algorithm to compute the length spectrum for hyperbolic 3-manifolds. We revisit their tiling algorithm to make it faster and suitable for verified computations (see Section~\ref{sec:verified}); we also generalize it for a wide range of new applications. These range from finding disjoint and embedded cusp neighborhoods to solving the isomorphism problem for closed hyperbolic 3-manifolds and cusped hyperbolic $n$-manifolds. We list these applications in Section~\ref{sec:applications}. This is followed by a review of the Hodgson--Weeks tiling algorithm in Section~\ref{sec:prevTiling} and a discussion how the new algorithm differs from it in Section~\ref{sec:newTilingIntro}. We give an overview of the paper in Section~\ref{sec:overview}.

\subsection{Applications} \label{sec:applications}

We now list applications of the new tiling algorithm (which itself is described in Section~\ref{sec:tiling}). In this paper, we discuss the first two of these applications (the maximal cusp area matrix and the Epstein-Penner decomposition) in detail. In forthcoming work \cite{ghht:lenSpec,goerner:drilling}, we discuss some of the other applications. We have also implemented many of these applications in SnapPy \cite{SnapPy} with an option to switch to verified computation using interval arithmetic.

\subsubsection{Computing the maximal cusp area matrix} \label{sec:maxCuspArea}

In Section~\ref{sec:compMaxCuspAreaMatrix}, we use the new tiling algorithm to compute a new invariant, the \emphasisText{maximal cusp area matrix}. It completely characterizes the configuration space of disjoint and embedded cusp neighborhoods.
\begin{definition} \label{def:maxCuspAreaMatrix}
Let $M$ be a finite-volume, complete, orientable hyperbolic $3$-manifold with cusps labeled $0,\dots, n-1$. Let $A_{ij}\in\R^+$ be the number such that cusp neighborhoods $C_i$ and $C_j$ with areas $A(C_i)$ and $A(C_j)$ about cusp $i$ and $j$ are embedded (if $i=j$) or disjoint (otherwise) if and only if $
A(C_i) A(C_j) \leq A_{ij}$. We call the symmetric matrix $A_M=(A_{ij})$ the \emphasisDef{maximal cusp area matrix}.
\end{definition}

Figure~\ref{fig:magicMfdPolytope} shows an example. An important application of disjoint and embedded cusp neighborhoods is finding exceptional slopes. Previous work on classifying exceptional slopes includes \cite{threeChainLinkSlopes,fiveChainLinkSlopes,sevenChainLinkSlopes,altKnotSlopes,excepSlopeCensus}.
As we explain in Section~\ref{sec:revisitSix}, the maximal cusp area matrix can be used to check whether slopes fulfill the conditions of the $6$-Theorem \cite{Agol:Six,Lackenby:Six}.

\begin{figure}[ht]
\begin{center}
\begin{minipage}{5cm}
\includegraphics[width=4cm]{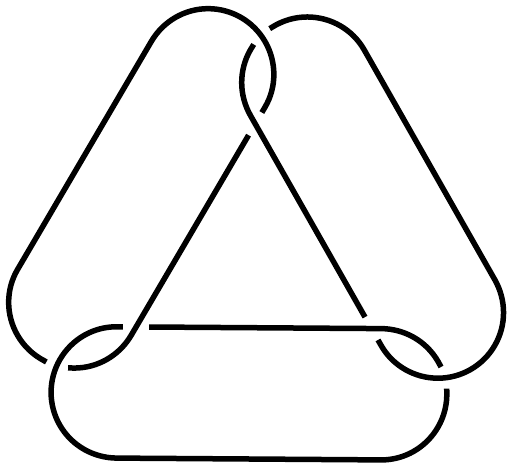}\\
\end{minipage} %
\begin{minipage}{6cm}
\begin{center}
\includegraphics[width=5.6cm]{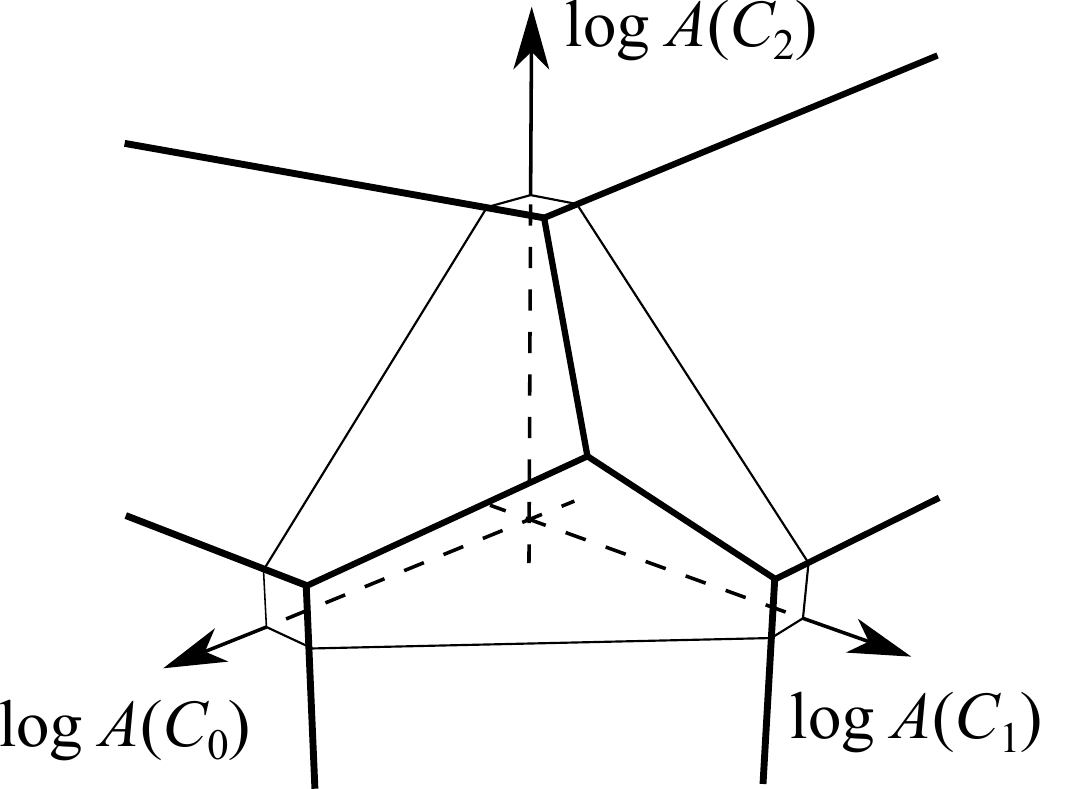}\\
~ \\
$A_M=\left(\begin{array}{rrr}
28 & 7 & 7\\
7 & 28 & 7 \\
7 & 7 & 28\end{array}\right)$~~~~~~~~~
\end{center}
\end{minipage}
\end{center}
\vspace{-0.5cm}
\caption{The 3-chain link $6^3_1$ (``magic manifold''), the corresponding maximal cusp area matrix\oxford{} and the polytope of disjoint and embedded cusp neighborhoods of the link complement. The SnapPy command for $A_M$ is: \texttt{Manifold(\discretionary{}{}{}"6\^{}3\_1").cusp\_area\_matrix()} (requires Version~3.2\protect\footnotemark or later).\label{fig:magicMfdPolytope}}
\end{figure}

\begin{remark}
The maximal cusp area matrix also determines the \emphasisText{maximal cusp volume}: the maximum possible total volume of disjoint and embedded cusp neighborhoods. The ratio of the maximal cusp volume to volume of a hyperbolic $3$-manifold is the \emphasisText{cusp density}; see \cite{adams:cuspDensity} for example. For the complement of $6^3_1$, the maximal cusp volume is realized at each of the four vertices of the polytope in Figure~\ref{fig:magicMfdPolytope}. The polytope for $7^3_1$ has the same combinatorics, but the maximal cusp volume is realized only at the central vertex. Note that computing the polytope, and then evaluating the volume at each vertex, is costly if there are many cusps. We thus pose the following challenge:
\begin{quote}
\emphasisText{Question:} Is there an efficient algorithm for maximal cusp volume?
\end{quote}
\end{remark}

\newpage

Previous techniques to find disjoint and embedded cusp neighborhoods suffer from the following limitations:
\begin{itemize}
\item The SnapPy kernel (which comes from Jeffrey Weeks' original SnapPea) provides a \emphasisText{stopper} for each cusp neighborhood which indicates when the cusp neighborhood would collide with itself or another cusp neighborhood. The cusp neighborhood view shows it as a colored bar limiting the slider for the cusp area. This gives partial information about the cusp area matrix. For example, a green stopper for a red cusp together with the areas of the green and red cusp determines the cusp area matrix entry for those two cusps. To determine the stoppers, the kernel requires the Epstein-Penner decomposition; it finds this using a heuristic process. There are examples (such as \texttt{o9\_17172}) where the kernel fails to recompute the stoppers when changing the cusp areas. Furthermore, the stoppers are not verified.
\item Ichihara and Masai \cite{altKnotSlopes} can verify that cusp neighborhoods are embedded and disjoint in the presence of an additional hypothesis. Namely, that the cusp neighborhoods intersect the given triangulation in standard form; see Figure~\ref{fig:cuspStdForm} here and Definition~\ref{def:stdForm} later. Thus, their choice of neighborhoods depends on the triangulation. Furthermore, their choice is not always maximal.
\end{itemize}

\footnotetext{The method \texttt{Manifold.cusp\_area\_matrix()} already existed in prior SnapPy versions. However, in prior versions and without further arguments, that method did \emphasisText{not} return the maximal cusp area matrix. Instead it returned lower bounds for the maximal cusp area matrix. These lower bounds were computed from cusp neighborhoods in standard form; see Figure~\ref{fig:cuspStdForm} and Definition~\ref{def:stdForm}.}

\begin{figure}[ht]
\begin{center}
\begin{minipage}{5.8cm}
\begin{center}
\includegraphics[scale=0.55]{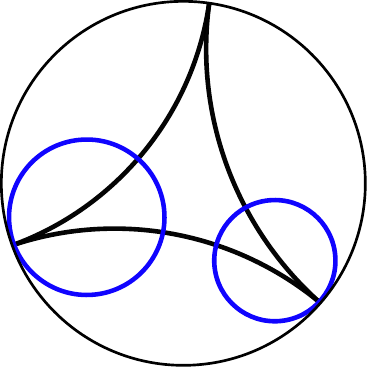}\\
In standard form
\end{center}
\end{minipage} %
\begin{minipage}{5.8cm}
\begin{center}
\includegraphics[scale=0.55]{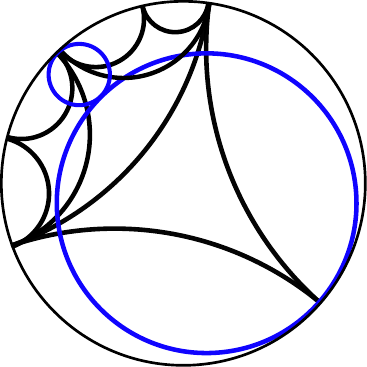}\\
Not in standard form
\end{center}
\end{minipage}
\end{center}
\vspace{-0.5cm}
\caption{Cusp neighborhoods intersecting an ideal triangulation. The picture shows the Poincar\'e disc model. For simplicity, we often show the picture for 2-dimensional instead of 3-dimensional triangulations.\label{fig:cuspStdForm}}
\end{figure}

In Section~\ref{sec:stdFormArgs}, we explain the method of \cite{altKnotSlopes} since we need it as a boot-strapping process; see Remark~\ref{rem:bootstrap}. When the cusp neighborhoods are not in standard form, they can intersect in the manifold even though they do not intersect along the edges. In this case, their method does not work. Thus, for generic cusp neighborhoods, we need to tile to see their intersection; again, see Figure~\ref{fig:cuspStdForm}.

\subsubsection{Solving the isomorphism problem for cusped hyperbolic $n$-manifolds}

In Section~\ref{sec:epAlgorithm}, we use the tiling algorithm to compute the Epstein-Penner decomposition (see \cite{EP}) of an oriented, finite-volume cusped hyperbolic $3$-manifold $M$. As a special case, we compute the canonical cell decomposition: the Epstein-Penner decomposition where all cusp neighborhoods have the same volume. In Section~\ref{sec:higherDim}, we generalize to higher dimensions. Thus, Sections~\ref{sec:epAlgorithm} and~\ref{sec:higherDim} settle a challenge posed by Sakuma and Weeks in \cite[Section~4]{sakumaWeeks:tilt}:
\begin{quote}
Find an algorithm which computes the canonical cell decomposition of any cusped hyperbolic $n$-manifold in a finite number of steps.
\end{quote}

They asked this because Weeks' algorithm to compute the Epstein-Penner decomposition is particular to dimension 3 and is heuristic: it performs $2\rightarrow 3$ and $3\rightarrow 2$ moves and randomizes when it gets stuck; also see \cite[Proposition~5.5]{weeks:canonical}.

We can turn the canonical cell decomposition into a triangulation by, for example, taking the barycentric subdivision. In particular, we have:
\begin{theorem}
There is an algorithm that, given geometric triangulations of oriented, finite-volume cusped hyperbolic $n$-manifolds $M$ and $M'$ (encoded, for example, by a Gram matrix with exact algebraic expressions for each simplex), decides whether $M$ and $M'$ are homeomorphic.
\end{theorem}
\begin{proof} Use Algorithm~\ref{algorithm:epAlgorithm} to compute the canonical cell decomposition of $M$ and $M'$. Canonically triangulate each. Check whether the triangulations are combinatorially isomorphic.
\end{proof}

Even better than this, we obtain a complete invariant for cusped hyperbolic $n$-manifolds. This is the isomorphism signature (see \cite{burton:encode} and, for its generalization to higher dimensions, \cite{Regina}) of the barycentric subdivision of the canonical cell decomposition. SnapPy's \emphasisText{isometry signature} chooses a more compact representation by skipping the subdivision if all canonical cells are tetrahedral; if some are not tetrahedral, then it uses a coarsening of the barycentric subdivison (see \cite{SnapPy} or \cite[Definition~3.3]{FGGTV:TetrahedralCensus}).

\subsubsection{Computing the length spectrum} \label{sec:motivationLenSpec}

In \cite{ghht:lenSpec}, we use the new tiling algorithm to accelerate and verify the computation
of the length spectrum. As we explain in Section~\ref{sec:prevTiling}, the Achilles heel of the existing length spectrum algorithm~\cite{hwcensus} is its dependence on the Dirichlet domain. We show how to overcome this weakness in Section~\ref{sec:newTilingIntro}. Doing so makes the new algorithm in \cite{ghht:lenSpec} verified in general and also much faster than \cite{hwcensus} in some cases: for example, 120ms vs 16h 53 to compute the systole of \texttt{o9\_00637}. See \cite{ghht:lenSpec} for details. 

\begin{remark}
In Section~\ref{sec:distStandard}, we also show how to use the new tiling algorithm to compute the embedding size of a geodesic tube or the distance between two closed geodesics. The work can be extended further to compute the ortho-length spectrum defined in \cite{meyerhoff:orthoSpec}. Unlike the length spectrum, the ortho-length spectrum completely determines a closed hyperbolic 3-manifold. However, the resulting invariant consists of algebraic numbers whereas the isometry signature is combinatorial. Thus, we focus on the isometry signature here and in \cite{goerner:drilling}.
\end{remark}


\subsubsection{Drilling simple closed geodesics}

In \cite{goerner:drilling}, we give the first algorithm
to drill any collection of disjoint simple closed geodesics. It uses the new tiling algorithm to compute a lower bound for the embedding radius of a tube about a geodesic; see Section~\ref{sec:lowerBoundSystem}. The lower bound is needed to safely perturb a geodesic that would otherwise intersect the 1-skeleton.


The existing drilling method
\cite[Section~2]{hwcensus} is restricted to \emphasisText{dual curves}: curves embedded in the dual 1-skeleton of the given triangulation. Furthermore, it does not check that the given dual curve is isotopic to a simple closed geodesic. For example, \texttt{Manifold("m115")\discretionary{}{}{}.drill(11)} drills a curve homotopic but not isotopic to a simple closed geodesic; see \cite{goerner:drilling} for details.

\subsubsection{Solving the isomorphism problem for closed hyperbolic 3-manifolds}

Using the techniques from the previous two sections, \cite{goerner:drilling} extends the isometry signature to closed oriented hyperbolic 3-manifolds. The new invariant of, say, $M$ is obtained from the isometry signatures of the cusped manifolds obtained from the drilling parents of $M$ coming from a canonical subset of the length spectrum.

\subsubsection{Computing the Margulis number for a hyperbolic 3-manifold}

We can also use the new tiling algorithm to determine whether a given number $\varepsilon>0$ is a Margulis number for a (closed or cusped) hyperbolic 3-manifold $M$. Recall that the $\varepsilon$-thin part of $M$ is the union of all essential loops of length less than $\varepsilon$ (or, equivalently, consists of all points with injectivity radius less than $\varepsilon/2$). We call $\varepsilon$ a \emphasisText{Margulis number} for $M$ if the $\varepsilon$-thin part is a union of embedded and disjoint tubes about simple closed geodesics and cusp neighborhoods.

Accordingly, fix $\varepsilon>0$. Suppose that $\gamma$ is a geodesic in $M$. The \emphasisText{$\varepsilon$-thin neighborhood} of $\gamma$ is the tube about $\gamma$ with radius given by \cite[Proposition~3.10]{fps:margulis}. Suppose that $c$ is a cusp of $M$ with shape $s$ (see Section~\ref{sec:peripheralConventions}). The \emphasisText{$\varepsilon$-thin neighborhood} of $c$ is the cusp neighborhood of area $ \ImPart(s) (h/l)^2$. Here $h=2\sinh(\varepsilon/2)$ is the Euclidean distance measured along a horosphere $H$ between two points on $H$ with hyperbolic distance $\varepsilon$; $l=\min_{(p,q)\in\Z^2\setminus 0} |p + qs|$ is the length of the shortest slope in $c$ (when the area is $\ImPart(s)$).

We can determine whether $\varepsilon$ is a Margulis number for $M$ as follows. For each geodesic $\{\gamma_i\}$ up to real length $\varepsilon$ and for each cusp, compute its $\varepsilon$-thin neighborhood. Use the techniques in Section~\ref{sec:embeddingSizes} to check whether these neighborhoods are embedded and disjoint in $M$.

One can now compute the optimal Margulis number $\mu(M)$ using the above procedure and bisection. However, we give a more efficient algorithm in \cite{ghht:lenSpec}.

\subsubsection{Ray-tracing the inside view}

The new tiling algorithm allows us to draw, for the first time, objects such as tubes about geodesics in the inside view of a given hyperbolic 3-manifold; see Remark~\ref{rem:localViewRaytracing}. Figure~\ref{fig:insideView} shows an example. The inside view is also a helpful tool for debugging the implementations of the algorithms discussed here and in forthcoming work.

\begin{figure}[ht]
\begin{center}
\includegraphics[width=9cm]{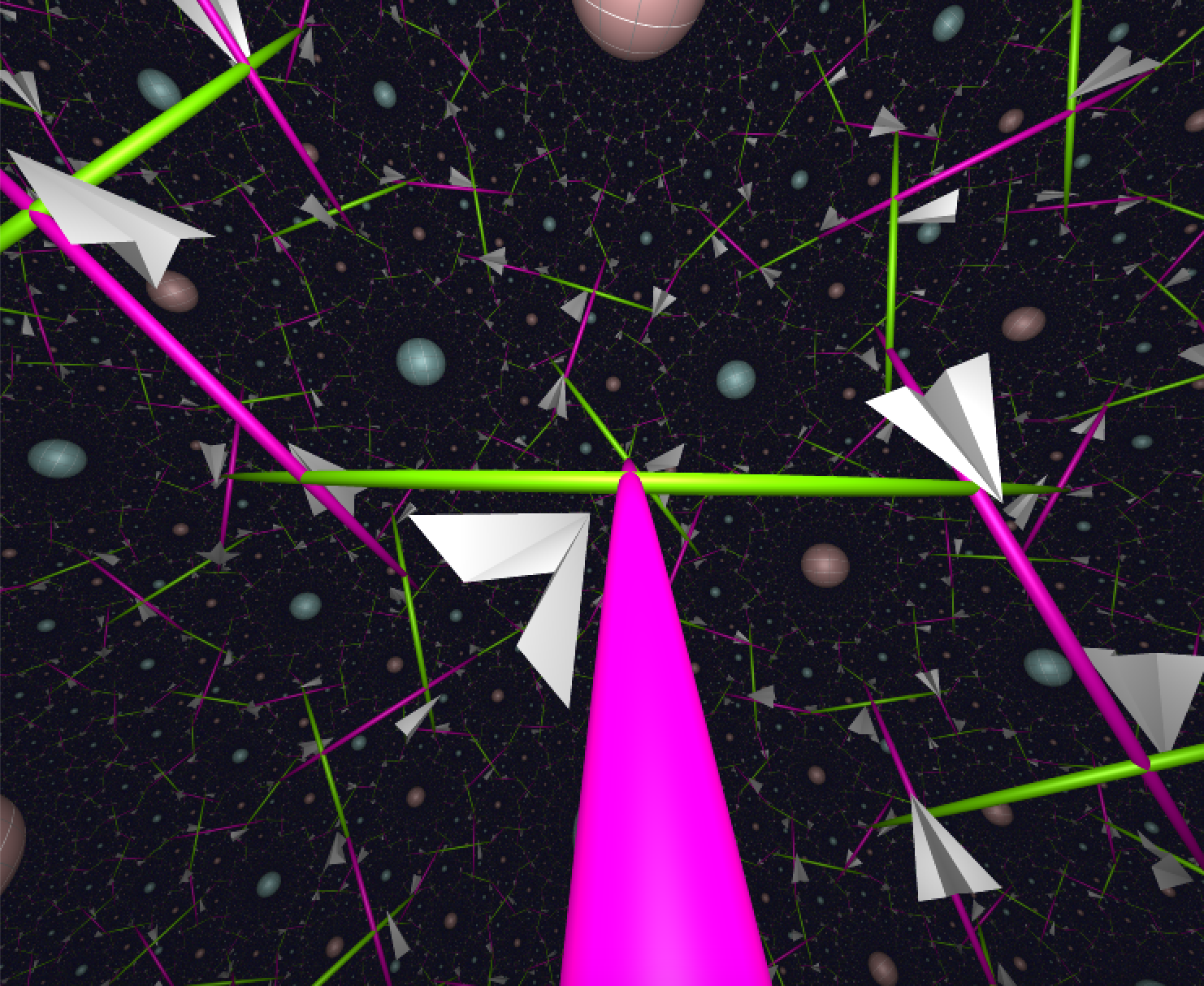}
\end{center}
\caption{Inside view of \texttt{m125} in SnapPy. It shows cusp neighborhoods, the two geodesics of length $1.0612\dots - 2.2370\dots i$\oxford{} and the camera as a paper plane.
The SnapPy command is: \texttt{Manifold("m125")\discretionary{}{}{}.inside\_view(\discretionary{}{}{}geodesics=["Ad","BDaDa"])}.\label{fig:insideView}}
\end{figure}


\subsection{Previous tiling algorithms} \label{sec:prevTiling}

We briefly review the tiling algorithm by Hodgson--Weeks \cite{hwcensus}. 
Let $P\subset \H^3$ be a fundamental polyhedron for the oriented, finite-volume hyperbolic $3$-manifold $M\cong \Gamma\backslash\H^3$.
Let $x\in P$ be a base point. Given a tiling radius $r\geq 0$, the goal is to find all translated tiles $mP$ whose translated basepoint $mx$ satisfies $d(x,mx)\leq r$. This can be used to compute, for example, the length spectrum up to a given cut-off length (which is related to $r$ by \cite[Proposition~3.5]{hwcensus}).

\begin{figure}[ht]
\begin{center}
\begingroup%
  \makeatletter%
  \providecommand\color[2][]{%
    \errmessage{(Inkscape) Color is used for the text in Inkscape, but the package 'color.sty' is not loaded}%
    \renewcommand\color[2][]{}%
  }%
  \providecommand\transparent[1]{%
    \errmessage{(Inkscape) Transparency is used (non-zero) for the text in Inkscape, but the package 'transparent.sty' is not loaded}%
    \renewcommand\transparent[1]{}%
  }%
  \providecommand\rotatebox[2]{#2}%
  \newcommand*\fsize{\dimexpr\f@size pt\relax}%
  \newcommand*\lineheight[1]{\fontsize{\fsize}{#1\fsize}\selectfont}%
  \ifx\svgwidth\undefined%
    \setlength{\unitlength}{381.22800812bp}%
    \ifx\svgscale\undefined%
      \relax%
    \else%
      \setlength{\unitlength}{\unitlength * \real{\svgscale}}%
    \fi%
  \else%
    \setlength{\unitlength}{\svgwidth}%
  \fi%
  \global\let\svgwidth\undefined%
  \global\let\svgscale\undefined%
  \makeatother%
  \begin{picture}(1,0.35769149)%
    \lineheight{1}%
    \setlength\tabcolsep{0pt}%
    \put(0,0){\includegraphics[width=\unitlength,page=1]{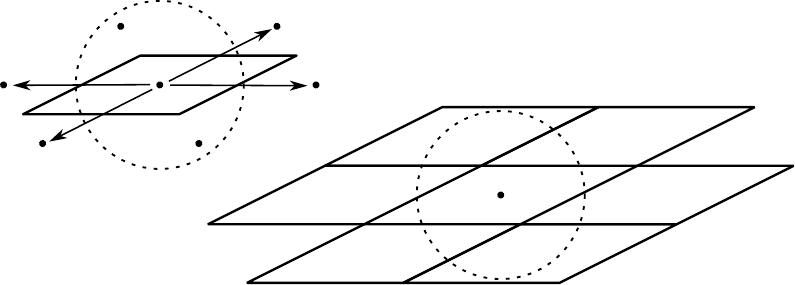}}%
    \put(0.62335199,0.11953216){\makebox(0,0)[lt]{\lineheight{1.25}\smash{\begin{tabular}[t]{l}$x$\end{tabular}}}}%
    \put(0.72546752,0.04381576){\makebox(0,0)[lt]{\lineheight{1.25}\smash{\begin{tabular}[t]{l}$r$\end{tabular}}}}%
    \put(0.1937611,0.25795795){\makebox(0,0)[lt]{\lineheight{1.25}\smash{\begin{tabular}[t]{l}$x$\end{tabular}}}}%
    \put(0.29615545,0.18244135){\makebox(0,0)[lt]{\lineheight{1.25}\smash{\begin{tabular}[t]{l}$r$\end{tabular}}}}%
    \put(0.00508223,0.11425977){\makebox(0,0)[lt]{\lineheight{1.25}\smash{\begin{tabular}[t]{l}Hodgson-Weeks\end{tabular}}}}%
    \put(0.55799624,0.23777083){\makebox(0,0)[lt]{\lineheight{1.25}\smash{\begin{tabular}[t]{l}New tiling algorithm\end{tabular}}}}%
  \end{picture}%
\endgroup%

\end{center}
\caption{An example of tiling with the fundamental polyhedron $P$ of a Euclidean torus. The Hodgson-Weeks recursively applies the face-pairing matrices and stops when $d(x,mx)>r$. In this case, it stops immediately and, thus, produces only $\Id P(=P)$ and misses the two non-trivial translates with $d(x,mx)\leq r$. The new algorithm stops when $d(x,mP)>r$. It produces all translates with $d(x,mP)\leq r$.\label{fig:tilingFund}}
\end{figure}
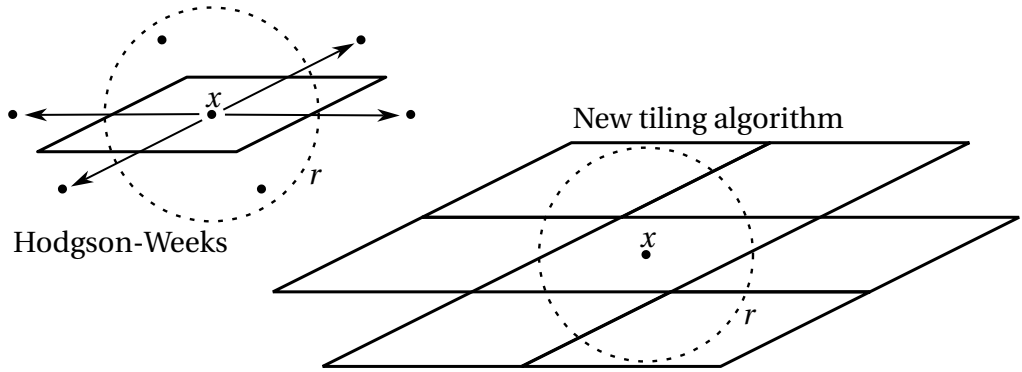

The Hodgson--Weeks algorithm starts with the single tile $\Id P(=P)$ and applies the face-pairing matrices recursively. It stops recursing when $d(x,mx)>r$. If $P$ is a Dirichlet domain about the base point $x$, the algorithm correctly finds all translates with $d(x,mx)\leq r$. However, if $P$ is not a Dirichlet domain, the algorithm can miss translates; see Figure~\ref{fig:tilingFund}.

\newpage

Using a Dirichlet domain $D$ can cause problems. For example:
\begin{itemize}
\item When $D$ has many faces, it maybe difficult to find $D$.
\item When $D$ has large spine radius, the length spectrum becomes very expensive (see \cite{ghht:lenSpec} for examples).
\item When $D$ has finite vertices of valence 4 or higher, it is impossible to verify $D$ with interval arithmetic.
\end{itemize}
In particular, we have the following conjecture.
\begin{conjecture}
Every Dirichlet domain of the census manifolds \texttt{m125} and \texttt{m129} has a finite vertex of valence four.
\end{conjecture}

\subsection{The new tiling algorithm} \label{sec:newTilingIntro}

As in the previous section, for our new tiling algorithm, we replace $d(x,mx)$ by $d(x,mP)$. Doing so leads to the following new features:
\begin{itemize}
\item We can work with any fundamental polyhedron $P$, not just a Dirichlet domain $D$; again see Figure~\ref{fig:tilingFund}.
\item Given a tiling radius $r\geq 0$, the algorithm produces tiles $mP$ covering the closed ball $\overline{B}_r(x)\subset \H^3$, not just the set $\overline{B}_r(x)\cap \Gamma x$. This is useful for applications such as determining whether $B_r(x)$ embeds into $M$. See \cite{ghht:lenSpec} for implications to the length spectrum computation.
\item In particular, we can use tetrahedral tiles. That is, we tile $\H^3$ with translates $mT_t$ of tetrahedra rather than translates $mP$ of the fundamental polyhedron $P=\cup T_t$. Again see \cite{ghht:lenSpec} for implications to the length spectrum computation.
\end{itemize}

\newpage

Other new features include the following:
\begin{itemize}
\item In addition to tiling about a point, we can now tile about a closed geodesic or cusp neighborhood. More precisely, let $K\subset\H^3$ be a lift of the object to be tiled. The tiles cover a closed neighborhood $\overline{B}_r(K)$ in the quotient space $\Gamma_K\backslash\H^3$ where $\Gamma_K\subset\Gamma$ is the stabilizer of $K$; see Section~\ref{sec:tilingGeneralCase}.
\item The algorithm produces a stream of tiles and tiling radii rather than taking a tiling radius as input. That is, each tile $t$ in the stream comes with a lower bound on the radius of the neighborhood of $K$ covered by the tiles up to $t$. In other words, we can think of the algorithm as a generator that can be queried for the next pair consisting of a tiling radius and a tile.
\item The algorithm supports verified computation using intervals; see Sections~\ref{sec:verified} and~\ref{sec:verifiedTiling}. In particular, the input itself (a fundamental polyhedron) can be verified with interval arithmetic; see Section~\ref{sec:geomTrig}.
\end{itemize}

\subsection{Overview} \label{sec:overview}

We review the semantics of verified computations in Section~\ref{sec:verified}, expressions for distances in hyperbolic space in Section~\ref{sec:hypSpace}\oxford{} and geometric structures on triangulations and developing a fundamental polyhedron $P$ in Section~\ref{sec:geomTrig}.

In Section~\ref{sec:algoInput}, we discuss how to represent the object to tile about so that it can serve as input to the tiling algorithm.

In Section~\ref{sec:graphTrace}, we introduce the graph tracing algorithm to seed the tiling algorithm.

At the heart of the paper is the tiling algorithm in Section~\ref{sec:tiling}.

It is followed by two sections on implementation details: dictionaries suitable for verified computations are introduced in Section~\ref{sec:verifiedDicts} and specialized to detect duplicate tiles in Section~\ref{sec:setLiftedTets}.

As an application of the tiling algorithm, Section~\ref{sec:embeddingSizes} focuses on the embedding size of an object or the distance between two objects in a hyperbolic 3-manifold $M$. In particular, Section~\ref{sec:compMaxCuspAreaMatrix} applies this to cusp neighborhoods to compute the maximal cusp area matrix.

As applications of the maximal cusp area matrix, we show how to find embedded and disjoint cusp neighborhoods in Section~\ref{sec:revisitSix}. In this section, we also review the $6$-Theorem to find exceptional slopes.

As another application of the tiling algorithm, we propose an algorithm to compute the Epstein-Penner decomposition in Section~\ref{sec:epAlgorithm}.

We conclude by generalizing to higher dimensions in Section~\ref{sec:higherDim}.

Appendix~\ref{App:Notation} lists the notation used throughout this paper.

\subsection{Acknowledgements}

The author thanks Neil R. Hoffman, Robert C. Haraway III and Maria Trnkova for starting the project to compute the verified length spectrum and Saul Schleimer for his detailed feedback on this paper. The author also thanks David Futer, Henry Segerman\oxford{} and Christian Zickert for many helpful discussions. The author is grateful to Marc Culler and Nathan Dunfield for our continuing collaboration on SnapPy.

\section{Verified computations} \label{sec:verified}

\subsection{Semantics of verified computations}

With the exception of Section~\ref{sec:epAlgorithm}, all algorithms in this paper can be made \emphasisText{verified} by using interval arithmetic. In general, we call a computation \emphasisText{verified} if it either returns intervals containing the true values or explicitly fails to indicate insufficient precision. By increasing precision, we can make the computation succeed eventually and return arbitrary tight intervals. For example, a verified computation of the maximal cusp area matrix $A_M$ gives an interval for each matrix entry $A_{ij}$ containing the true value of $A_{ij}$.  Analogously, a verified computations of a geometric structure on an ideal triangulation $\myTrig$ gives intervals for the shapes $z_0, \dots, z_{n-1}$ containing the unique solution to the logarithmic gluing equations (see Section~\ref{sec:geometricStruct} for further discussion). For finite triangulations, we can give a hyperbolic structure by assigning lengths to the edges subject to a system of equations and inequalities. This system does not have a unique solution and a verified computation gives intervals containing some solution. Recall that the output of the tiling algorithm is a stream of tiles. Thus, the semantics of its verified computation is slightly more complicated and explained in detail in Section~\ref{sec:verifiedTiling}.

Typically, computations for hyperbolic manifolds start with shapes or edge lengths. These are given implicitly as solution to a system of equations. Thus, finding verified intervals for shapes or edge lengths requires methods such as the interval Newton method or Krawczyk test. Applying these techniques to hyperbolic manifolds was pioneered by \cite{HIKMOT} and extended to finite triangulations by the author \cite{matthiasVerifyingFinite}.

In many subsequent computations (including the ones in this paper), new quantities are given explicitly in terms of already computed quantities. We assume that the interval implementation of each operator fulfills the inclusion principle:
\begin{definition} \label{def:inclusionPrinciple}
Let $f\from U\to V$ be a function with $U\subset \R^n$ and $V\subset \R^m$. The interval version $F$ of $f$ maps $n$ floating-point intervals to $m$ floating-point intervals. It fulfills the \emphasisDef{inclusion principle} if for every tuple $X$ of floating-point intervals contained in $U$ and every $x\in X$, we have $f(x)\in F(X)$.
\end{definition}
Then, inductively and starting with the shapes or edge lengths, all computed intervals contain the true values (obtained when using exact arithmetic) of the respective quantity.

\subsection{Interval arithmetic conventions} \label{sec:intConv}

We write most algorithms in this paper so that we can switch between unverified and verified computations by just switching the real number type between floating-point numbers and intervals. For verified computations, we can set $\varepsilon$'s to zero, for example, in Algorithm~\ref{algorithm:graphTracingVerificationStep}. We do not give an unverified implementation of some data structures such as the set $S$ in the tiling algorithm (Algorithm~\ref{algorithm:tilingAny}).

We use $\underline{a}$ and $\overline{a}$ to denote the left and right endpoint of an interval $a=[\underline{a},\overline{a}]$. When switching to unverified computations, $\underline{a}$ and $\overline{a}$ simply denote $a$ itself.

We assume that the interval operator $[\underline{a},\overline{a}] < [\underline{b},\overline{b}]$ is true if $\overline{a} < \underline{b}$ and false otherwise (and analogously for other inequality operators)\footnote{This matches the ``<'' operator of SageMath's \texttt{RealIntervalField} \cite{SageMath}.}. In other words, $a<b$ should be read as ``the interval estimates prove that $a<b$''. Note that if $a<b$ is false for the interval estimates, it does not prove that $a\geq b$. So an algorithm also needs to explicitly also check for $a\geq b$ if it relies on it (see $\Eq_b$ in Section~\ref{sec:equalityPred} for an example which we use in Section~\ref{sec:gammaDict} to solve the word problem for hyperbolic manifolds). This is not the only possible implementation of the interval comparison operator. The conservative choice is making the operator explicitly fail unless we can decide whether $a<b$ is true or false, that is either $\overline{a} < \underline{b}$ or $\underline{a} \geq \overline{b}$.

For convenience, when computing a lower bound $l$ for, say, $x$, we also use an interval $L$. In that case, $\underline{L}$ is a lower bound for (the true value of) $x$ and $\overline{L}$ is meaningless. An example is the lower bound $r_k$ for the radius of the neighborhood covered by the tiles computed by Algorithm~\ref{algorithm:tilingAny}. Upper bounds work analogously. This avoids using two different types for real numbers and, thus, comparisons between floating point numbers and intervals (which some interval arithmetic implementations such as the one in SageMath \cite{SageMath} disallow).


\section{Hyperbolic space} \label{sec:hypSpace}

\subsection{Hyperboloid model} \label{sec:hyperboloidModel}

We use the hyperboloid model of hyperbolic $3$-space. An introduction can be found in \cite{PurcellKnotTheory,ratcliffe:hyp}. We use the following notation:
\begin{itemize}
\item $x\cdot x'$ is the inner product of signature $(-,+,+,+)$.
\item $\H^3$ is the hyperboloid model $\left\{ x \in \R^{1,3} : x_0>0~\mbox{and}~x\cdot x = -1\right\}$. 
\item $\posLightCone$ is the future light-cone $\left\{x \in \R^{1,3} : x_0>0~\mbox{and}~x\cdot x=0\right\}$. A ray in $\posLightCone$ corresponds to an ideal point of $\H^3$.
\item $B(l)$ is the horoball $\left\{x\in\H^3 : x\cdot l > -1\right\}$ for a light-like $l\in\posLightCone$.
\item $\widehat{x}$ or $(x)\postNorm$ is the projection $x/\sqrt{-x\cdot x}$ of a time-like $x\in\R^{1,3}$ onto $\H^3$.
\item $\hypPlane{n}$ is the plane $n^\perp\cap\H^3=\left\{x\in\H^3 : x\cdot n = 0\right\}$ defined by a space-like normal vector $n\in\R^{1,3}$ with $n\cdot n = 1$.
\item $x_0x_1$ is the line given by $\gamma(t)=((1-t)x_0 + tx_1))\postNorm$ for $x_0, x_1\in\posLightCone$.\end{itemize}

\subsection{Distances} \label{sec:hypDistances}

Using the techniques described later in Section~\ref{sec:distanceProofs}, we obtain the following expressions for the (unsigned) distances between points $x$ and $x'$, lines $x_0x_1$ and $x'_0x_1'$ and planes $\hypPlane{n}$ and $\hypPlane{n'}$:
\allowdisplaybreaks
\begin{align*}
d(x, x')&=\cosh^{-1}\left(-x\cdot x'\right) = \cosh^{-1}\left( \frac{(x'-x)\cdot (x'-x)}{2}+1\right)\\
d(x, x_0x_1) &= \cosh^{-1}\sqrt{-2\frac{(x_0\cdot x)(x\cdot x_1)}{x_0\cdot x_1}}\\
\displaybreak\\
d(x_0x_1,x'_0x'_1)&=2 \sinh^{-1}\sqrt{ \frac{ \lambda - 1}{2}}\\
&\text{where}\quad\lambda = \sqrt{\frac{(x_0\cdot x'_0)(x_1\cdot x'_1)}{(x_0\cdot x_1)(x'_0\cdot x'_1)}} + \sqrt{\frac{(x_0\cdot x'_1)(x_1\cdot x'_0)}{(x_0\cdot x_1)(x'_0\cdot x'_1)}}\\
d(x, \hypPlane{n}) &= \left|\sinh^{-1} x\cdot n\right| \\
d(x_0x_1,\hypPlane{n})&=\begin{cases}\displaystyle\sinh^{-1}\sqrt{-2 \frac{(x_0\cdot n)(n\cdot x_1)}{x_0\cdot x_1}} & \text{if}\quad (x_0\cdot n)(n\cdot x_1)> 0 \\ 0 & \text{otherwise}\end{cases}\\
d(\hypPlane{n}, \hypPlane{n'})&=\begin{cases}\displaystyle\cosh^{-1} \left| n\cdot n'\right| & \text{if}\quad \left|n\cdot n'\right | \geq 1\\ 0 & \text{otherwise}\end{cases}\\
\intertext{For the distance of a point $x$ to a horoball $B(l)$, we use the signed distance that is negative when $x$ is inside the horoball:}
d(x, B(l)) &= d_\text{unsigned}(x, B(l)) - d_\text{unsigned}(x, \H^3 \setminus B(l))\\
& = \log(-x\cdot l)
\intertext{In general, we define the distance to a horoball by $d(K, B(l)) = \inf_{x\in K} d(x, B(l))$ using extended reals. In particular, $d(B(l),B(l))$ is well-defined and $-\infty$. We have:}
d(x_0x_1, B(l)) & = \log \sqrt{-2\frac{(x_0\cdot l)(l\cdot x_1)}{x_0\cdot x_1}} \\
d(\hypPlane{n}, B(l)) & = \log \left|n\cdot l\right|\\
d(B(l), B(l')) & = \log\left(-\frac{l\cdot l'}{2}\right)
\intertext{The signed distance of the projections of points $x, x'$ onto the line $x_0x_1$ is given as follows. Note that the distance is positive if the projection of $x'$ is closer to $x_1$ than the projection of $x$:}
d_{x_0x_1}(x, x')&=\log\sqrt{\frac{(x\cdot x_1)(x'\cdot x_0)}{(x\cdot x_0)(x'\cdot x_1)}}
\intertext{For the midpoint of the line, this simplifies to:}
d_{x_0x_1}((x_0+x_1)\postNorm, x')&=\log\sqrt{\frac{x'\cdot x_0}{x'\cdot x_1}}
\intertext{The (Euclidean) length of the line segment $x_0x_1$ projected onto $\partial B(l)$ is given by}
d_{\partial B(l)}(x_0, x_1)&= \sqrt{-2\frac{x_0\cdot x_1}{(x_0\cdot l)(l\cdot x_1)}}
\end{align*}

\begin{remark}
The above expressions also apply to $\H^n$.
\end{remark}

Let $T$ be an ideal triangle in the plane $\hypPlane{n}$ spanned by $v_0, v_1, v_2\in\posLightCone$. For $\{i,j,k\}=\{0,1,2\}$, let
\[
m_i = \left(- (v_i\cdot v_k) v_j - (v_i\cdot v_j) v_k\right)\postNorm
\]
be the point where the incircle of $T$ touches the line $v_jv_k$ and let $n_i$ be the normalized projection of $v_i-m_i$ onto the tangent space of $\H^3$ at $m_i$. Then the planes $\hypPlane{n_i}$ in $\H^3$ are orthogonal to $\hypPlane{n}$ and bound $T$ so that we can compute the distances to $T$ as follows:
\begin{align*}
d(x, T)&=\begin{cases}d(x, v_iv_j) & \text{if}\quad x\cdot n_k < 0~\text{where}~\{i,j,k\}=\{0,1,2\}\\ d(x, \hypPlane{n}) & \text{otherwise}\end{cases}\\
d(B(l), T)& =  \begin{cases}d(B(l), v_iv_j) & \text{if}\quad l\cdot n_k < 0~\text{where}~\{i,j,k\}=\{0,1,2\}\\ d(B(l), \hypPlane{n}) & \text{otherwise}\end{cases}\\
d(x_0x_1, T)&=\begin{cases}d(x_0x_1, v_iv_j) & \text{if}\quad p\cdot n_k< 0~\text{where}~\{i,j,k\}=\{0,1,2\}\\ d(x_0x_1,\hypPlane{n}) & \text{otherwise}\end{cases}\\
& \text{where $p=(\left|x_1\cdot n\right|x_0+\left|x_0\cdot n\right|x_1)\postNorm$ is the point on $x_0x_1$ closest to $\hypPlane{n}$} \\
\end{align*}

\subsection{Derivations of distances} \label{sec:distanceProofs}

We now sketch out how to derive the above expressions.

Recall that $r_d=(\cosh d, \sinh d, 0, 0)$ parametrizes a line in $\H^3$ by length.

For $d(x,x')$: Without loss of generality, assume $x=r_0=(1,0,0,0)$ and $x'=r_d$. Then $d(x,x')=d$ and $x\cdot x' = -\cosh d$, so we get the above expression.

For $d(x, \hypPlane{n})$: Without loss of generality, assume $x=r_d$ and $n=(0,1,0,0)$.
 
 For $d(x,x_0x_1)$: Fix $x_0$ and $x_1$. Knowing $d(x, B(x_0))$ narrows down the possible locations of $x$ to a horosphere. Knowing both $d(x, B(x_0))$ and $d(x, B(x_1))$ narrows down $x$ to the intersection of two horosphere which is a circle concentric about $x_0x_1$ with radius $d(x,x_0x_1)$. If we also know $d(B(x_0), B(x_1))$, we can compute this radius. In other words, $d(x,x_0x_1)$ is a function of $x\cdot x_0$, $x\cdot x_1$ and $x_0\cdot x_1$. Being invariant under scaling $x_0$ or $x_1$, $d(x, x_0x_1)$ must be a function of $(x_0\cdot x)(x\cdot x_1)/(x_0\cdot x_1)$. Without loss of generality, assume  $x=r_d$ and $x_0, x_1=(1,0,\pm 1,0)$. The point $p$ on $x_0x_1$ closest to $x$ is $(1,0,0,0)$ and $d(x,p)=d$.

For $d(x_0x_1, x'_0x'_1)$: We can apply an $\SO(1,3)$ transformation such that
\begin{eqnarray*}
x_0 & = & (\cosh d/2, +\sin\alpha,+\cos\alpha,+\sinh d/2) \lambda_0\\
x_1 & = & (\cosh d/2, -\sin\alpha, -\cos\alpha,+\sinh d/2) \lambda_1\\
x'_0 & = & (\cosh d/2, +\sin\alpha,-\cos\alpha,-\sinh d/2) \lambda'_0\\
x'_1 & = & (\cosh d/2, -\sin\alpha, +\cos\alpha,-\sinh d/2) \lambda'_1
\end{eqnarray*}
where $d$ is the distance between the lines. We can compute the inner products of the four points and solve for $d$.

For $d(x_0x_1,\hypPlane{n})$: If $x_0$ and $x_1$ are on two different sides of the plane $\hypPlane{n}$, then the distance is zero. Otherwise, assume $x_0, x_1 = (\cosh d, \sinh d, \pm 1, 0)$ and $n=(0,1,0,0)$. The points on the line and plane closest to each other are $r_d$ and $r_0$ and $d(r_d,r_0)=d$. 

For $d(\hypPlane{n},\hypPlane{n'})$: If $\hypPlane{n}$ and $\hypPlane{n'}$ intersect in $\H^3$, assume their intersection contains the origin $(1,0,0,0)$ such that $n$ and $n'$ are in the Euclidean subspace $\R^3$ of $\R^{1,3}$. Otherwise, consider $n=(0,1,0,0)$ and $n'=(\sinh d, \cosh d, 0, 0)$. The points on the two planes closest to each other are $r_0$ and $r_d$.

\subsection{Verified computation} \label{sec:verifiedDistances}

The expressions in Section~\ref{sec:hypDistances} must be modified to be robust when using intervals. For example, the interval estimate for $-x\cdot x'$ can contain values $v<1$ even though we know that the true value is $\geq 1$. In particular, this can happen if $x$ and $x'$ are interval estimates for the same point. Since $\cosh^{-1}(v)$ is not defined in this case, we need to modify the expression for $d(x,x')$ as follows:
\[
d(x, x')=\cosh^{-1}\left(-x\cdot x' \cap [1, \infty) \right)
\]
Another example is $d(\hypPlane{n}, \hypPlane{n'})$:
\[
d(\hypPlane{n}, \hypPlane{n'})=\cosh^{-1}\left(\max(1, |n\cdot n'|)\right)
\]
To ensure a non-trivial lower bound for $d(x,x_0x_1)$ in the degenerate case where $x_0$ and $x_1$ coincide and have overlapping intervals, we use:
\[
d(x, x_0x_1) = \cosh^{-1}\sqrt{[1,\infty)\cap\frac{2(x_0\cdot x)(x\cdot x_1)}{[0,\infty)\cap -x_0\cdot x_1}}
\]

Consider evaluating $d(x, T)$ from interval estimates. Assume that the interval estimates are not good enough to either prove $x\cdot n_k< 0$ for one $k$ or prove $x\cdot n_k \geq 0$ for all $k$. Then we need to take the smallest interval containing the interval for $d(x, \hypPlane{n})$ and the interval for each $d(x,v_iv_j)$ where the interval estimates could not prove $x\cdot n_k \geq 0$. However, if we are only interested in a lower bound for $d(x, T)$ (for example, in Algorithm~\ref{algorithm:tilingAny}), we can simply return $d(x, v_iv_j)$ if the interval estimates could prove $x\cdot n_k\leq 0$ for any k and $d(x, \hypPlane{n})$ otherwise.

During interval computations, it is often better to avoid or postpone normalization. In particular, we do not assume the ideal endpoints of a line $x_0x_1$ are scaled to achieve $x_0\cdot x_1 = -1$.

\section{Geometric triangulations} \label{sec:geomTrig}

\subsection{Triangulations} \label{sec:triangulations}

Let $\myTrig$ be an oriented, 3-dimensional triangulation with tetrahedra $T_0,\dots, T_{n-1}$. We call a vertex \emphasisText{finite} if its link is a 2-sphere. Otherwise, we call the vertex \emphasisText{ideal}. We denote by $M$ the manifold obtained by removing the ideal vertices from the geometric realization of $\myTrig$. Note that $M$ is closed if all vertices are finite. Otherwise, $M$ has cusps corresponding to the ideal vertices.

Unless stated otherwise, we use $t$ and $t'$ to index tetrahedra $T_t$ and $T_{t'}$. We use $v$ and $v'$ and $f$ and $f'$ to index the vertices and faces of a tetrahedron, respectively.

If face $f$ of tetrahedron $T_t$ is glued to $T_{t'}$, we denote this by $n_t^f=t'$ (for neighbor). Furthermore, we denote the face-gluing permutation for face $f$ of tetrahedron $T_t$ by $\sigma_t^f\in S_4\setminus A_4$. That is, vertex $v\not=f$ of face $f$ of $T_t$ is taken to vertex $\sigma_t^f(v)$ of face $\sigma_t^f(f)$ of $T_{t'}$.

\subsection{Geometric structure} \label{sec:geometricStruct}

Consider an assignment of an isometry class of positively oriented, non-flat geodesic tetrahedra in $\H^3$ to each $T_0, \dots, T_{n-1}$. The vertices of the tetrahedron can be finite, ideal or a mix. We call such an assignment a \emphasisText{geometric structure} on $\myTrig$ if the tetrahedra glue up to form a hyperbolic structure on $M$. If this structure is incomplete, we assume that it can be completed by attaching circles to $M$. We call these circles \emphasisText{core curves}. We call the resulting $M_\text{filled}$ the \emphasisText{filled manifold}. A geometric structure yields a conjugacy class of representations $\pi_1(M)\to\pi_1(M_\text{filled})\to\SO(1,3)$.

Geometric structures can be encoded in various ways. Here, we are primarily focusing on cross ratios $z_0,\dots, z_{n-1}\in\C\setminus\{0,1\}$ (also called shapes) parameterizing ideal geodesic tetrahedra (though many parts of this paper also apply for finite geodesic tetrahedra). The gluing equations introduced by Thurston \cite[Section~4.2]{ThurstonNotes} tell us when they form a geometric structure (also see \cite[\hbox{E.6-i}]{BenedettiPetronio} and \cite[Section~4]{PurcellKnotTheory}). In this paper, we assume that we can compute arbitrary tight interval estimates for the cross ratios. This has been pioneered by \cite{HIKMOT} (relying on \cite[Lemma~2.4]{moser} without saying so; see \cite[Appendix]{matthiasVerifyingFinite}) and implemented in SnapPy by the author since.

\begin{remark}
It is still unproven that every cusped hyperbolic $3$-manifold $M$ has an ideal triangulation admitting a geometric structure. \cite{PetronioWeeks:partiallyFlatTrig} show that every $M$ has an ideal triangulation admitting a semi-geometric structure where some but not all tetrahedra are allowed to be flat. Note that exact arithmetic can be used to verify a semi-geometric structure and that the tiling algorithm (Algorithm~\ref{algorithm:tilingAny}) can be modified to take those as input.
\end{remark}

\begin{remark}
While all manifolds in the SnapPy orientable cusped census are hyperbolic, the provided triangulation is not geometric in six cases: \texttt{t02774}, \texttt{t03695}, \texttt{t06262}, \texttt{t06288}, \texttt{t06361}, \texttt{t07404}.
\end{remark}

\begin{remark}
Alternatives to encoding a geometric structures by shapes include:
\begin{itemize}
\item Ptolemy coordinates parametrizing ideal geodesic tetrahedra (with some extra structure); see \cite{GTZ:ptolemyCoords,GGZ:PtolemyField}.
\item Edge lengths parametrizing finite geodesic tetrahedra; see \cite{casson:Geo} and \cite{matthiasVerifyingFinite} (for verified computation).
\item Vertex Gram matrices (generalizing edge lengths) parameterizing generalized tetrahedra; see \cite{heardThesis}.
\end{itemize}
\end{remark}

\subsection{Cusp neighborhoods}

All \emphasisText{cusp neighborhoods} are assumed to be \emphasisText{horocusp neighborhoods}. That is, a cusp neighborhood is given by an open horoball $H=B(l)$ in the universal cover $\H^3$ where $l\in\posLightCone$ corresponds to the considered cusp. Note that the Decktransformations act on $H$. If any two images of $H$ under this action are either identical or disjoint, we say that the cusp neighborhood \emphasisText{embeds}. Similarly, we define when two cusp neighborhoods are \emphasisText{disjoint}. The area $A(C)$ and the volume of a cusp neighborhood $C$ are those of the quotient of $H$ by the Decktransformations fixing $H$. Note that the volume of a cusp neighborhood $C$ is $A(C)/2$.

\subsection{Cusp cross section} \label{sec:cuspCrossSec}

Consider a complete cusp. The \emphasisText{cusp triangulation} is the link of the ideal vertex corresponding to the cusp. A \emphasisText{cusp cross section} is an assignment of lengths to the edges of the cusp triangulation that is compatible with the shapes of the tetrahedra. More precisely, the ratio of two edge lengths in a triangle is given by the absolute value of the respective cross ratio $z_t$, $1/(1-z_t)$ or $1-1/z_t$.

The possible cusp cross sections are in one-to-one correspondence to the possible cusp neighborhoods about the cusp. Both are parametrized by area. The correspondence is given by the edge lengths of the horotriangle when intersecting the cusp neighborhood with a tetrahedron $T_t$ extended by all rays starting at the relevant vertex and intersecting the tetrahedron. We can easily develop some cusp cross section by assigning some length to some edge of the cusp triangulation and recursing to determine all other edge lengths. Once developed, we can normalize a cusp cross section to a particular area.

\subsection{Standard form, disjointness and embeddedness} \label{sec:stdFormArgs}

We show how to verify that cusp neighborhoods are disjoint and embedded assuming they are in standard form which was already illustrated in Figure~\ref{fig:cuspStdForm} and is defined as follows:
\begin{definition} \label{def:stdForm}
A horoball $H$ about a vertex of an ideal geometric tetrahedron $T$ intersects the tetrahedron in \emphasisDef{standard form} if $H$ intersects $T$ in three of its four faces. A cusp neighborhood intersects an ideal geometric triangulation $\myTrig$ in \emphasisDef{standard form} if, for each tetrahedron $T_t$, the respective horoballs about the vertices of $T_t$ intersect $T_t$ in standard form.
\end{definition}

Let $\myTrig$ be an ideal geometric triangulation of an oriented, complete hyperbolic 3-manifold $M$. We can check whether a cusp neighborhood intersects $\myTrig$ in standard form by checking that the triangles in the corresponding cusp cross section all have areas less than the maximal area in the following lemma:
\begin{lemma} \label{lemma:maxAreaStandardForm}
Let $T$ be an ideal geodesic tetrahedron with cross ratio $z$. Pick a vertex $v$ of $T$. Let $H$ be any horoball about $v$ intersecting $T$ in standard form. The maximal area of the intersection of $\partial H$ with $T$ across all such horoballs is
\[
A(z)=A\left(\frac{1}{1-z}\right)=A\left(1-\frac{1}{z}\right)=\frac{\ImPart(z)}{2 h(z)^2}
\]
\[
\quad\text{where}~h(z)=\frac{1}{2}\begin{cases}
|z|\cdot|1-z|/\ImPart(z) & \text{if $\Delta$ is acute} \\
\text{longest edge of $\Delta$} & \text{otherwise}
\end{cases}
\]
and where $\Delta$ is the Euclidean triangle $0, 1, z$ in $\C$.
\end{lemma}
\begin{proof}
A calculation using that $h(z)$ is the Euclidean height of the hyperbolic triangle $0,1,z$ in the upper half space model. That is, $h(z)$ is the circumradius of $\Delta$ when $\Delta$ is acute.
\end{proof}

\begin{figure}[ht]
\begin{center}
\begingroup%
  \makeatletter%
  \providecommand\color[2][]{%
    \errmessage{(Inkscape) Color is used for the text in Inkscape, but the package 'color.sty' is not loaded}%
    \renewcommand\color[2][]{}%
  }%
  \providecommand\transparent[1]{%
    \errmessage{(Inkscape) Transparency is used (non-zero) for the text in Inkscape, but the package 'transparent.sty' is not loaded}%
    \renewcommand\transparent[1]{}%
  }%
  \providecommand\rotatebox[2]{#2}%
  \newcommand*\fsize{\dimexpr\f@size pt\relax}%
  \newcommand*\lineheight[1]{\fontsize{\fsize}{#1\fsize}\selectfont}%
  \ifx\svgwidth\undefined%
    \setlength{\unitlength}{124.98516173bp}%
    \ifx\svgscale\undefined%
      \relax%
    \else%
      \setlength{\unitlength}{\unitlength * \real{\svgscale}}%
    \fi%
  \else%
    \setlength{\unitlength}{\svgwidth}%
  \fi%
  \global\let\svgwidth\undefined%
  \global\let\svgscale\undefined%
  \makeatother%
  \begin{picture}(1,0.8314165)%
    \lineheight{1}%
    \setlength\tabcolsep{0pt}%
    \put(0,0){\includegraphics[width=\unitlength,page=1]{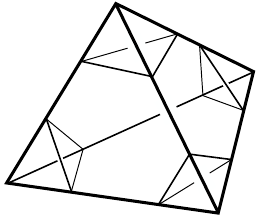}}%
    \put(0.41235917,0.50802155){\makebox(0,0)[lt]{\lineheight{1.25}\smash{\begin{tabular}[t]{l}$a$\end{tabular}}}}%
    \put(0.60622999,0.14135283){\makebox(0,0)[lt]{\lineheight{1.25}\smash{\begin{tabular}[t]{l}$b$\end{tabular}}}}%
    \put(0.71580915,0.30572157){\makebox(0,0)[lt]{\lineheight{1.25}\smash{\begin{tabular}[t]{l}$e$\end{tabular}}}}%
  \end{picture}%
\endgroup%

\end{center}
\caption{An ideal geodesic tetrahedron truncated by horoballs about vertices. The signed distance between the horoballs at the ends of edge $e$ is $-\log(ab)$ where $a$ and $b$ are the Euclidean lengths measured along the corresponding horospheres. In particular, the horoballs are disjoint if and only if $ab\leq 1$. Note the similarity to the Ptolemy coordinate for $e$ given by $c=1/\sqrt{ab}$ (see \cite{zickert:volume}) when using complex cusp cross sections.\label{fig:truncTet}}
\end{figure}
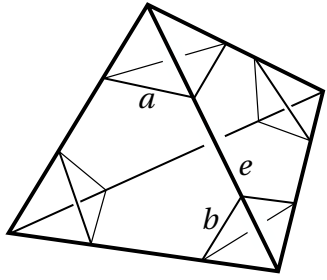
Let $C_i$ be a cusp neighborhood for cusp $i$ intersecting $\myTrig$ in standard form. Consider the corresponding cusp cross sections. If the condition in Figure~\ref{fig:truncTet} is true for each edge $e$ of $\myTrig$ connecting cusp $i$ to itself, then $C_i$ is embedded. Similarly, two cusp neighborhoods $C_i$ and $C_j$ intersecting $\myTrig$ in standard form each are disjoint if the condition is true for each edge $e$ connecting cusp $i$ and $j$.

\subsection{Charts} \label{sec:trigCharts}

For many computations, we need to assign not just an isometry class but an actual geodesic tetrahedron to each $T_t$. We call a triangulation $\myTrig$ together with such an assignment forming a hyperbolic structure on $M$ a \emphasisText{triangulation with charts}. This is in analogy to charts forming an atlas to define a higher structure on a topological manifold. We abuse notation and also denote by $T_t$ the geodesic tetrahedron in $\H^3$. A triangulation with charts yields an $\SO(1,3)$-representation of the fundamental groupoid associated with the dual $2$-skeleton of $\myTrig$. Here, the \emphasisText{fundamental groupoid} associated with cell complex $C$ is the category where the objects are the vertices of $C$ and the morphisms are paths of edges of $C$ up to homotopy. 


We store the following information for each tetrahedron $T_t$:
\begin{itemize}
\item The vertices $\vec{v}_t^v\in\H^3\cup\posLightCone$ with $v=0,\dots,3$ spanning the geodesic tetrahedron.
\item The triangle $T_t^f$ spanned by $\vec{v}_t^v$ for $v\not=f$ and forming face $f$ with $f=0,\dots, 3$.
\item The outward-facing normal vector $\vec{n}_t^f$ for $f=0,\dots, 3$ such that $T_t^f$ is supported by $\hypPlane{\vec{n}_t^f}$. We have $\vec{n}_t^f \cdot \vec{v}_t^{v} = 0$ for $v\not=f$, $\vec{n}_t^f\cdot v_t^{f} < 0$ and $\vec{n}_t^f\cdot \vec{n}_t^f = 1$.
\item The face-pairing matrices $g_t^f$ for each face $f=0,\dots, 3$ taking the vertices to the vertices of the neighboring tetrahedron $T_{t'}$ with $t'=n_t^f$. More precisely, for each $v\not=f$, there is a $\lambda>0$ with
$g_t^f \vec{v}_t^v = \lambda \vec{v}_{t'}^{v'}$ where $v'=\sigma_t^f(v).$ If vertex $v$ of $T_t$ corresponds to complete cusp, we require $\lambda=1$. Note that $\big(g_t^f\big)^{-1}=g_{t'}^{f'}$ where $t'=n_t^f$ and $f'=\sigma_t^f(f)$.
\end{itemize}

We can use the shapes $z_0,\dots,z_{n-1}$ to find the values for the vertices $\vec{v}^v_t$ up to a scaling. This uniquely determines the face-pairing matrices. In practice, it might be easier to do the computation in the upper half-space model with boundary $\C P^1$ and convert to the hyperboloid model later. For verified computations, we can represent points in $\C P^1$ in two ways: as a complex interval for $z$ with a flag selecting one of the charts $(1,z)$ and $(z,1)$; or as algebraic datatype $\C P^1=\C \cup \{\infty\}$ that is either a complex interval for $z\in\C$ or a sentinel for $\infty$  (suitable when starting with a vertex at $\infty$ to develop). For a complete cusp, we can use a cusp cross section to scale each involved $\vec{v}^v_t$ as follows. We pick a face of the tetrahedron $T_t$ adjacent to $\vec{v}^v_t$ and let $e$ be the length of the corresponding edge in the cusp cross section and $\vec{v}_t^{v'}$ and $\vec{v}_t^{v''}$ the other two vertices of the face. Scale $\vec{v}^v_t$ such that
\[
e=d_{\partial B(\vec{v}_t^v)}\left(\vec{v}_t^{v'},\vec{v}_t^{v''}\right).
\]
\label{sec:addedCuspCrossSec}

\begin{remark} \label{rem:insideViewCharts}
When ray-tracing the inside view, we run into the limitation that only single precision floating point numbers are available on some graphics card. To avoid large numerical errors, it is advisable to pick charts such that the incenter of each tetrahedron is at the origin $(1,0,0,0)\in\H^3$. This is also used in \cite{cohomologyFractals} to render cohomology fractals.
\end{remark}

\subsection{Developed fundamental polyhedron} \label{sec:devFundPoly}

Consider a triangulation $\myTrig$ with charts. We call it a \emphasisText{developed fundamental polyhedron} if the tetrahedra form a fundamental polyhedron $P=\cup T_t$ in $\H^3$ for $M\subset M_\text{filled}$. In other words, if there is a spanning tree in the dual $1$-skeleton such that each face-pairing matrix corresponding to an edge in the spanning tree is the identity. The edges in the complement of the spanning tree in the dual $1$-skeleton give elements in the fundamental group. They form the generators of the face-pairing representation $\pi_1(P)$. Some edges give the same element in $\pi_1(P)$ and can be skipped as generators. A developed fundamental polyhedron yields a representation $\rho\from\pi_1(P)\to\SO(1,3)$ factoring through $\pi_1(M_\text{filled})$. We denote its image by $\Gamma\subset\SO(1,3)$, so $M_\text{filled}\cong \Gamma\backslash\H^3$. Note that the face-pairing representation $\pi_1(P)$ is canonically isomorphic to $\pi_1(M,x)$ for any $x$ in the interior of the fundamental polyhedron. This is because there is a unique path (up to homotopy) connecting any two such $x$ inside the fundamental polyhedron.

\subsubsection{Developing using cocycles} \label{sec:cocycleDev}

In the ideal hyperbolic 3-dimensional case, it is straightforward to develop a fundamental polyhedron by computing the fourth vertex of an ideal tetrahedron from its cross ratio and the other three vertices. However, this does not generalize to other cases and higher dimensions. In this section, we show a technique that does generalize. For concreteness, we do this first for a finite 3-dimensional triangulation $\myTrig$ with a hyperbolic structure given by edge lengths and discuss generalizations later in Remark~\ref{rem:cocycleIdeal} and Section~\ref{sec:higherDim}.

We continue where \cite[Section~3]{matthiasVerifyingFinite} left off and use the same conventions. In particular, $\H^3=\{z+t \jmath : t > 0\}$ denotes the upper half-space model for the remainder of this section. Let $\myTrig$ be an oriented, finite triangulation with edge parameters fulfilling all conditions of \cite[Theorem~3.5]{matthiasVerifyingFinite}. Let $\rho$ be the resulting representation of the fundamental groupoid for $\hat{\myTrig}$. That is, given a path $\gamma$ of edges in $\hat{\myTrig}$, $\rho(\gamma)$ is the product of matrices assigned to those edges by the cocycle on $\hat{\myTrig}\subset\myTrig$.

\begin{figure}[ht]
\begin{center}
\begingroup%
  \makeatletter%
  \providecommand\color[2][]{%
    \errmessage{(Inkscape) Color is used for the text in Inkscape, but the package 'color.sty' is not loaded}%
    \renewcommand\color[2][]{}%
  }%
  \providecommand\transparent[1]{%
    \errmessage{(Inkscape) Transparency is used (non-zero) for the text in Inkscape, but the package 'transparent.sty' is not loaded}%
    \renewcommand\transparent[1]{}%
  }%
  \providecommand\rotatebox[2]{#2}%
  \newcommand*\fsize{\dimexpr\f@size pt\relax}%
  \newcommand*\lineheight[1]{\fontsize{\fsize}{#1\fsize}\selectfont}%
  \ifx\svgwidth\undefined%
    \setlength{\unitlength}{255.33697173bp}%
    \ifx\svgscale\undefined%
      \relax%
    \else%
      \setlength{\unitlength}{\unitlength * \real{\svgscale}}%
    \fi%
  \else%
    \setlength{\unitlength}{\svgwidth}%
  \fi%
  \global\let\svgwidth\undefined%
  \global\let\svgscale\undefined%
  \makeatother%
  \begin{picture}(1,0.5174361)%
    \lineheight{1}%
    \setlength\tabcolsep{0pt}%
    \put(0,0){\includegraphics[width=\unitlength,page=1]{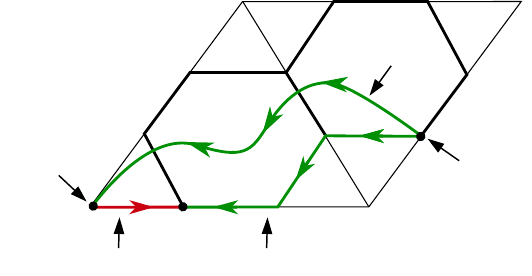}}%
    \put(0.86222926,0.18790697){\makebox(0,0)[lt]{\lineheight{1.25}\smash{\begin{tabular}[t]{l}$x$\end{tabular}}}}%
    \put(-0.00281123,0.20617573){\makebox(0,0)[lt]{\lineheight{1.25}\smash{\begin{tabular}[t]{l}$v=p(\tilde{v})$\end{tabular}}}}%
    \put(0.70506091,0.41411006){\makebox(0,0)[lt]{\lineheight{1.25}\smash{\begin{tabular}[t]{l}$p(\gamma)$\end{tabular}}}}%
    \put(0.19286705,0.01203447){\makebox(0,0)[lt]{\lineheight{1.25}\smash{\begin{tabular}[t]{l}$\gamma_v$\end{tabular}}}}%
    \put(0.44163269,0.01150519){\makebox(0,0)[lt]{\lineheight{1.25}\smash{\begin{tabular}[t]{l}$[\gamma_v p(\gamma)]$\end{tabular}}}}%
  \end{picture}%
\endgroup%

\end{center}
\vspace{-0.6cm}
\caption{Paths to compute the vertex developing map $\Psi\from\widetilde{V}\to\overline{\H}^3$.\label{fig:devPath}}
\end{figure}

Let $\widetilde{V}\subset\widetilde{\myTrig}$ be the vertices of the universal cover $p\from\widetilde{\myTrig}\to\myTrig$. We obtain a \emphasisText{vertex developing map $\Psi\from\widetilde{V}\to\overline{\H}^3$} as follows:  Pick a base vertex $x$ of $\hat{\myTrig}$ and a lift $\widetilde{x}\in\widetilde{\myTrig}$ of $x$. Pick a point in $q\in\overline{\H^3}$. For each vertex $v$ of $\myTrig$, pick a path $\gamma_v$ from $v$ to a vertex of $\hat{\myTrig}$ that is not passing through $\hat{\myTrig}$. Now, given a vertex $\tilde{v}\in\widetilde{V}$, let $\gamma$ be a path in $\widetilde{\myTrig}$ from $\tilde{x}$ to $\tilde{v}$ and let $v=p(\tilde{v})$. Let $[~\_~]$ denote homotoping a path to a path of edges of $\hat{\myTrig}$. We set $\psi(\tilde{v})=\rho([\gamma_v p(\gamma)])q$; see Figure~\ref{fig:devPath}.

If we pick $q=\jmath\in\H^3$, $\Psi$ is (the restriction to the vertices of) a geometric developing map. In this case, a different choice of $x$ or $\widetilde{x}$ changes $\Psi$ by conjugation and the choices of $\gamma_v$ do not affect $\Psi$ at all.

In general, $\Psi$ is equivariant with respect to a geometric representation (which is unique up to conjugation). Thus, $\Psi$ gives the same invariant in the respective group homology and, in particular, the same volume and Chern-Simons invariant as a geometric representation.

\begin{remark}
Let $\myTrig$ be a finite triangulation with a hyperbolic structure. We can compute the volume (and even Chern-Simons invariant) without invoking a formula for the volume of a finite simplex (such as in \cite{Ushijima:volFinite}) by picking $q\in\CP^1$: for each tetrahedron $T_t$ of $\myTrig$, $\Psi$ now gives us ideal vertices and thus a cross ratio $z_t$; the volume is $\sum D(z_t)$ where $D$ is the Bloch-Wigner dilogarithm.
\end{remark}

For a fundamental polyhedron $P$ defined by a spanning tree in the dual $1$-skeleton, we proceed similarly using the complex $\overline{P}$ obtained by replacing each tetrahedron $T_t$ in $P$ by a doubly truncated tetrahedron as in \cite[Figure~2]{matthiasVerifyingFinite}. Pick a base vertex $x$ of $\overline{P}$ and $q=\jmath$. This gives positions for the vertices of $P$. For a face-pairing of $P$, the corresponding matrix is given by $\rho(\gamma')^{-1}\rho(\gamma)$ where $\gamma$ and $\gamma'$ are paths from $x$ to vertices $v$ and $v'$ of $\overline{P}$ paired by the face-pairing.

\begin{remark} \label{rem:cocycleIdeal}
This works analogously for the cocycles coming from cross ratios (see \cite[Example~10.15]{GGZ:GluingEquations} and \cite[Section~3.2]{marche:CS}) or Ptolemy coordinates (see \cite[Figure~3]{GGZ:PtolemyField} and \cite[Proposition~10.4]{GGZ:GluingEquations}). Note that if $\myTrig$ has ideal vertices, $\widetilde{\myTrig}$ is obtained by temporarily removing the vertices of $\myTrig$, taking the universal cover and adding the vertices back in for each end. We also need to pick $q=0$ or $q=\infty$ (depending on conventions) to obtain the geometric developing map.
\end{remark}

\section{Input to the tiling algorithm} \label{sec:algoInput}

\subsection{Input to the algorithm}

The input to the tiling algorithm is a \emphasisText{standard geometric object} $K$ and a \emphasisText{lifted tetrahedron} $mT_t$ with $m\in\Gamma$ serving as seed. They are defined as follows:

\begin{definition}
Let $P\subset\H^3$ be a developed fundamental polyhedron for a hyperbolic 3-manifold $M$ with $M_\text{filled}\cong \Gamma\backslash\H^3$. We define a \emphasisDef{standard geometric object} $K$ as one of the following: \label{def:standardObject}
\begin{itemize}
\item A point $x\in\H^3$. 
\item A line $x_0x_1$ with $x_0, x_1\in\posLightCone$ forming a closed geodesic $\gamma$ in $M_\text{filled}$. We also store the corresponding $h\in\Gamma$ fixing $x_0x_1$.
\item A horoball $B(l)$ corresponding to a cusp neighborhood in $M$.
\end{itemize}
\end{definition}

We use $B_r(K)$ to denote the open neighborhood of size $r$ about a standard geometric object $K\subset\H^3$ of size $r$. We also use $\overline{B}_r(K)$ the corresponding closed neighborhood. If $K$ is a point $x\in\H^3$, this matches the standard definition of a ball $B_r(x)$ about $x$. If $K$ is a line, $B_r(K)$ is a tube about $K$ of radius $r$. If $K=B(l)$ is a horoball, we use the signed distance so $B_r(K)$ is another horoball $B(e^{-r} l)$. $B_r(K)$ is empty if $r=-\infty$ or if $K$ is a point or line and $r\leq 0$.

\begin{definition} \label{def:seed}
The lifted tetrahedra $mT_t$ is a \emphasisDef{seed} for the standard geometric object $K$ if
\begin{itemize}
\item In case $K$ is a point: $mT_t$ contains $K$.
\item In case $K$ is a line:
\begin{itemize}
\item If the corresponding geodesic $\gamma$ is not a core curve: $mT_t$ intersects $K$.
\item Otherwise: a vertex of $mT_t$ coincides with an endpoint of $K$.
\end{itemize}
\item In case $K$ is a horoball $B(l)$: $l$ is parallel to a vertex $m\vec{v}^v_t$ of $mT_t$.
\end{itemize}
\end{definition}

\subsection{Input from applications} \label{sec:inputFromApp}

An application can specify a standard geometric object $K$ in the following ways:
\begin{itemize}
\item In case $K$ is point: $x$ can be given as a point known to be inside a tetrahedron $T_t$. In this case, we can simply use $\Id T_t(=T_t)$ as seed. An example is the incenter of $T_t$.  If $x$ is any point in $\H^3$, we need to find a lifted tetrahedron $mT_t$ containing $x$ as a seed. We show how to do this in the next section.
\item In case $K$ is a line corresponding to a closed geodesic $\gamma$ in $M_\text{filled}$:
\begin{itemize}
\item If $\gamma$ is not a core curve, the application specifies a non-parabolic element $h\in\Gamma$ corresponding to $\gamma$. We compute the fixed points $x_0$ and $x_1$ of $h$ which span $K$. We use a point $x$ on the line as base point. We again use the technique in the next section to find a lifted tetrahedron $mT_t$ containing $x$ as a seed. Note that when we sample $x$ as $(\mu_0x_0+\mu_1x_1)\postNorm$, we should pick $\mu_0, \mu_1>0$ such that $\mu_0/\mu_1$ is not close to an algebraic number with small defining polynomial. This minimizes the chance that the base point hits the 1-skeleton of the geometric triangulation.
\item If $\gamma$ is a core curve, then the application specifies the ideal vertex of $\myTrig$ corresponding to the cusp that $\gamma$ completes. Let $T_t$ be a tetrahedron such vertex $v$ of $T_t$ corresponds to the given ideal vertex. We set $K$ to the lift of $\gamma$ which has one endpoint coinciding with the vertex $\vec{v}_t^v$ of $T_t$ and $h$ to the primitive element in $\Gamma$ fixing $K$. We can compute $K$ and $h$ using suitable peripheral curves in the corresponding cusp or using generalized cusp cross sections; see \cite{ghht:lenSpec}. We can then use the tetrahedron $T_t$ itself as seed $\Id T_t$.
\end{itemize}
\item In case $K$ is a horoball: we assume that $K$ is a horoball $B(\vec{v}^v_t)$ about a vertex of the developed fundamental polyhedron $P$. Thus, we can use the tetrahedron $T_t$ itself as seed $\Id T_t$.
\end{itemize}

\newpage

\section{Graph tracing} \label{sec:graphTrace}

In this section, we show how to find one or several lifted tetrahedra $mT_t$ such that at least one contains a given point $x\in\H^3$. These lifted tetrahedra serve as seeds (see Definition~\ref{def:seed}) for the tiling algorithm (Algorithm~\ref{algorithm:tilingAny}) later.

\subsection{Graph trace algorithms}

Graph tracing has a heuristic step (Algorithm~\ref{algorithm:graphTracing}) and a verification step (Algorithm~\ref{algorithm:graphTracingVerificationStep}). The verification step addresses, in particular, the case where the point $x$ is on a face of a tetrahedron. In this case, we cannot use interval arithmetic to prove that $x$ is contained in any of the two (closed) lifted tetrahedra $mT_t$ adjacent to the face. However, we can prove (with sufficient precision), that $x$ is in the union of those two neighboring tetrahedra (more precisely, in the intersection of the six half-spaces defined by the six non-paired faces). Such a pair of candidate seeds is sufficient as input for the tiling algorithm.

\begin{example}
There are cases where having the point $x$ on a face is unavoidable. Consider a shortest geodesic $\gamma$ in the geometric census triangulation \texttt{m125}. $\gamma$ is entirely contained in the $2$-skeleton. Thus, any point $x$ we pick on a lift $K$ of $\gamma$ (as in Section~\ref{sec:inputFromApp}) is on a face of a tetrahedron. Note though that a generic choice of $x$ on a geodesic avoids the $1$-skeleton.
\end{example}

\begin{algorithm}
\begin{tabular}{rp{15.4cm}}
{\bf Input:} & Developed fundamental polyhedron $P\subset \H^3$.\\
& Point $x\in\H^3$.\\
{\bf Output:} & Lifted tetrahedron $mT_t$ ``close'' to $x$ (that is, we cannot prove that $x$ is outside of $mT_t$ if using intervals).\\
{\bf Note:} & $\varepsilon = 0$ if using intervals.\\
& $\varepsilon >0$ if using floating-point numbers to avoid potential infinite loop when $x$ is on a face of a tetrahedron and the algorithm does not settle on either of the two tetrahedra adjacent to that face.\\
{\bf Algorithm:}
\end{tabular}
\begin{algorithmSteps}
\item Pick tetrahedron $T_t$. \label{step:pickTet}
\item $m\leftarrow\Id$ and $f\leftarrow -1$.
\item Repeat:
\begin{algorithmSteps}
\item $d_{f'} \leftarrow x\cdot \left(m \vec{n}^{f'}_t\right)$ for $f' = 0, \dots, 3$.
\item If no $d_0, \dots, d_3 > \varepsilon$: Return $mT_t$.
\item $f\leftarrow$ index of $d_{f'}$ with largest value/midpoint where $f'\not=f$.
\item $t, f \leftarrow n_t^f, \sigma_t^f(f)$.
\item $m\leftarrow m g_t^f$
\end{algorithmSteps}
\end{algorithmSteps} 
\caption{Graph trace (heuristic).\label{algorithm:graphTracing}}
\end{algorithm}

\begin{algorithm}
\begin{tabular}{rp{15.4cm}}
{\bf Input:} & Developed fundamental polyhedron $P\subset \H^3$.\\
& Point $x\in\H^3$.\\
& Lifted tetrahedron $m{T_t}$ close to $x$ (output of Algorithm~\ref{algorithm:graphTracing}).\\
{\bf Output:} & Lifted tetrahedra $m_0T_{t_0},\dots,m_{k-1}T_{t_{k-1}}$, at least one contains $x$.\\
{\bf Note:} & $\varepsilon = 0$ if verified.\\
{\bf Algorithm:}
\end{tabular}
\begin{algorithmSteps}
\item $d_{f} \leftarrow x\cdot \left(m \vec{n}^{f}_t\right)$ for $f = 0, \dots, 3$.
\item If $d_0, \dots, d_3 < -\varepsilon$: Return $mT_t$.
\item If three of the four inequalities $d_0, \dots, d_3 < -\varepsilon$ hold:
\begin{algorithmSteps}
\item $f'\leftarrow$ the unique $d_{f'}$ where $d_{f'}$ does not hold.
\item $t', f' \leftarrow n_t^{f'}, \sigma_t^{f'}(f')$.
\item $m'\leftarrow m g_{t'}^{f'}$
\item $d'_f \leftarrow x\cdot \left(m'\vec{n}^{f'}_{t'}\right)$ for $f=0,\dots, 3$.
\item If $d'_0, \dots, \hat{d'}_{f'}, \dots, d'_3 < -\varepsilon$: Return $mT_t$ and $m'T_{t'}$.
\end{algorithmSteps}
\item Fail.
\end{algorithmSteps}
\caption{Graph trace (verification).\label{algorithm:graphTracingVerificationStep}}
\end{algorithm}

\subsection{Dual views} \label{sec:graphTraceDualViews}

Note that Algorithm~\ref{algorithm:graphTracing} finds a lifted tetrahedron $mT_t$ close to $x$. We call this the \emphasisText{object view} version of the graph trace algorithm.

Let $(K,x)$ be some geometric object with base point $x$. We can rewrite Algorithm~\ref{algorithm:graphTracing} so that it acts on $(K,x)$ by the inverse of $g^f_t$ instead of $m$ by $g^f_t$. Since $m$ is always the identity, we can just drop it from the algorithm and its output. We call this the \emphasisText{tetrahedra view} version of the graph trace algorithm.

\subsection{Camera in the inside view}

One application of the graph trace algorithm is the ray-traced inside view: we need to ensure that the origin $x$ of the camera is inside (or at least close) to a known tetrahedron $T_t$.

More precisely, we encode the camera as a frame $s\in\SO(1,3)$ in the coordinate system given by the chart of a tetrahedron $T_t$. That is, the columns of $s$ are the origin $x$ and the right, up, and forward axis of the camera. We apply the tetrahedra view version of Algorithm~\ref{algorithm:graphTracing} to $(s, x)$. This is also used in \cite[Section~7.4 and 7.8]{cohomologyFractals} to generate the images of the cohomology fractals (even though it is not explained there).

Recall from Remark~\ref{rem:insideViewCharts} that the tetrahedra in the ray-traced inside view do not form a fundamental polyhedron. Thus, we need to change the tetrahedra view version of Algorithm~\ref{algorithm:graphTracing} to skip Step~\ref{step:pickTet} and instead take as input the tetrahedron $T_t$ whose chart we use as coordinate system for the camera $(s,x)$.

\clearpage

\section{Tiling algorithm} \label{sec:tiling}

Let $P=\cup T_t$ be a developed fundamental polyhedron (as defined in Section~\ref{sec:devFundPoly}) for a hyperbolic 3-manifold $M$. Let $\Gamma\subset\SO(1,3)$ be the group generated by the face-pairing matrices of $P$ such that $M_\text{filled}\cong \Gamma\backslash\H^3$. Let $K$ be a standard object (as defined in Definition~\ref{def:standardObject}). We call $mT_t$ with $m\in\Gamma$ a \emphasisText{lifted tetrahedron} (in $\H^3$). 

\subsection{Simple case} \label{sec:tilePoint}

Before tackling the general case, we make the following simplifying assumptions:
\begin{itemize}
\item We use exact arithmetic.
\item The standard object $K$ we tile about is a point $x\in\H^3$ in the interior of a lifted tetrahedron $m^\text{seed}T_{t^\text{seed}}$.
\item $M\cong\Gamma\backslash\H^3$ is complete.
\end{itemize}

\begin{figure}[ht]
\begin{center}
\scalebox{0.88}{
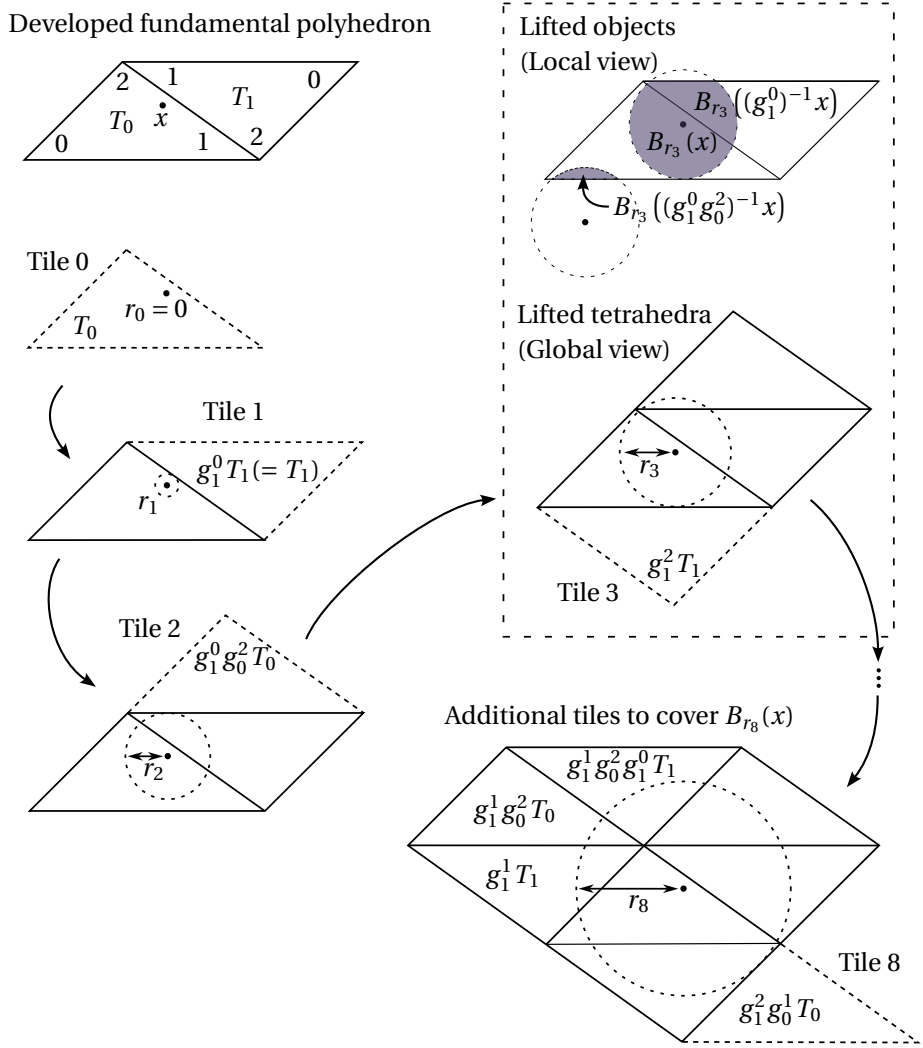}
\end{center}
\vspace{-0.5cm}
\caption{Tiling about a point $K=x$ for the developed fundamental polyhedron $P$ of a Euclidean torus. We can check whether the neighborhood $B_r(x)$ with $r=r_3$ embeds into the manifold $M$ by considering the first three tiles. That is, we check for each simplex $T_t$ whether the neighborhoods about the respective lifted objects $m_i^{-1}x$ are disjoint. For simplex $T_0$, we get two lifted neighborhoods, namely about $\Id x(=x)$ and $(g_1^0g_0^2)^{-1}x$. These two lifts are disjoint. For simplex $T_1$, we only have a single lift, namely about $(g_1^0)^{-1}x(=x)$, and, thus, have nothing to check. Therefore, $B_{r_3}(x)$ embeds into $M$. Note that $B_{r_8}(x)$ does not embed into $M$. Thus, the injectivity radius about $x$ (or, equivalently, its embedding size $d_M(x,x)/2$) is half the minimal distance between any two distinct lifted objects $m_0^{-1}x, \dots, m_7^{-1}x$ corresponding to the same tetrahedron.
\label{fig:tilingExample}}
\end{figure}

The lifted tetrahedra $mT_t$ tile all of $\H^3$. They are partially ordered by their distances $d(x,mT_t)$. We enumerate them, respecting this partial order, as follows. The first lifted tetrahedron is the seed $m^\text{seed}T_{t^\text{seed}}$. We obtain a subsequent lifted tetrahedron by iteratively picking a lifted tetrahedron closest to $x$, that is neighboring an already enumerated tetrahedron, and that has not been enumerated already.

Figure~\ref{fig:tilingExample} shows an example where $x\in T_0$ so that we can use $\Id T_0$ as the seed. 
This gives us a stream of pairs $(r_0, m_0T_{t_0}), (r_1, m_1T_{t_1}), \dots$ of distinct lifted tetrahedra and non-decreasing tiling radii $r_i=d(x,m_iT_{t_i})$ with $r_i\to\infty$.

\subsection{General case} \label{sec:tilingGeneralCase}

Note that the above procedure requires deciding which of two distances $d(x, mT_t)$ and $d(x, m'T_{t'})$ are larger. In particular, when they are equal, we must prove they are equal and then break a tie. Thus, we cannot use interval arithmetic for the above procedure.

Luckily, many applications do not require the tiles in exactly this order. Instead, the following weaker property for a stream of tiles $(r_0, m_0T_{t_0}), (r_1, m_1T_{t_1}), \dots$ is sufficient: for each $i$, the lifted tetrahedra $m_0T_{t_i},\dots, m_{i-1}T_{t_{i-1}}$ cover the closed ball $\overline{B}_{r_i}(x)$.

Loosening the requirement on the stream in this way allows us to use a lower bound for, rather than the exact value of, $d(x, mT_t)$. That is, we pick as the next lifted tetrahedron the one with lowest lower bound for $d(x,mT_t)$. This allows us to drop the first simplifying assumption and, thus, use interval arithmetic.

This weaker requirement also allows us to start with a list of candidate seeds (such as given by Algorithm~\ref{algorithm:graphTracingVerificationStep}) \[m^\text{seed}_0T_{t^\text{seed}_0},~\dots,~ m^\text{seed}_{k-1}T_{t^\text{seed}_{k-1}}\] where we know that at least one of them contains $x$ (or, more generally, is a seed for $K$) but we do not know which. That is, we simply enumerate all distinct candidate seeds first with $r_i=-\infty$.

Next, note that we need infinitely many lifted tetrahedra $mT_t$ in $\H^3$ to cover any non-empty neighborhood $\overline{B}_r(K)\subset\H^3$ of $K$ when $K$ is a line or a horoball. To drop the second simplifying assumption, we instead tile the quotient space $\Gamma_K\backslash\H^3$ where $\Gamma_K=\{m\in\Gamma : mK = K\}$ is the stabilizer of $K$. For the implementation, this means that we regard two lifted tetrahedra $mT_t$ and $m'T_{t'}$ as the same if $m$ and $m'$ are representatives of the same coset of $\Gamma_K\backslash\Gamma$ and $t=t'$.

The resulting stream of lifted tetrahedra $(r_0, m_0T_{t_0}), (r_1, m_1T_{t_1}), \dots$ now covers the neighborhoods $\overline{B}_{r_i}(K)$ of the object $K$ in the quotient space $\Gamma_K\backslash\H^3$. We call this stream the \emphasisText{object view}. We may convert this to the stream $(r_0, (m_0^{-1}K, T_{t_0})),\allowbreak (r_1, (m_1^{-1}K, T_{t_1})),\dots$ where we act on the object $K$ instead the tetrahedra $T_t$. We call this the \emphasisText{tetrahedra view} (similarly to Section~\ref{sec:graphTraceDualViews}). Note that a lifted object $(m^{-1}K, T_t)$ appears only once in the tetrahedra view when regarded as a pair of a subset of $\H^3$ and an index $t$. 

One application of the tetrahedra view is to compute distances between standard objects in a manifold. This is done in Section~\ref{sec:embeddingSizes}. See Figure~\ref{fig:tilingExample} for an example.

\begin{remark} \label{rem:localViewRaytracing}
Another application of the tetrahedra view is ray-tracing the inside view of a hyperbolic manifold $M$. This application requires specifying the lifted objects to trace against in each simplex (with some additional coordinate transforms in light of Remark~\ref{rem:insideViewCharts}). Imagine we want to show $\overline{B}_{r_3}(x)$ in the inside view corresponding to Figure~\ref{fig:tilingExample} from a camera inside $T_0$ looking at $(g_1^0g_0^2)^{-1}x$. If we did not give the shader for $T_0$ the lifted object corresponding to Tile~2, the rays from the camera would hit the lower edge of $T_0$ instead and we would see a hole in $\partial \overline{B}_{r_3}(x)$ in the inside view. For a 3-dimensional hyperbolic example; see Figure~\ref{fig:tilingRaytracing}.
\end{remark}

\begin{figure}[h]
\begin{center}
\includegraphics[width=5cm]{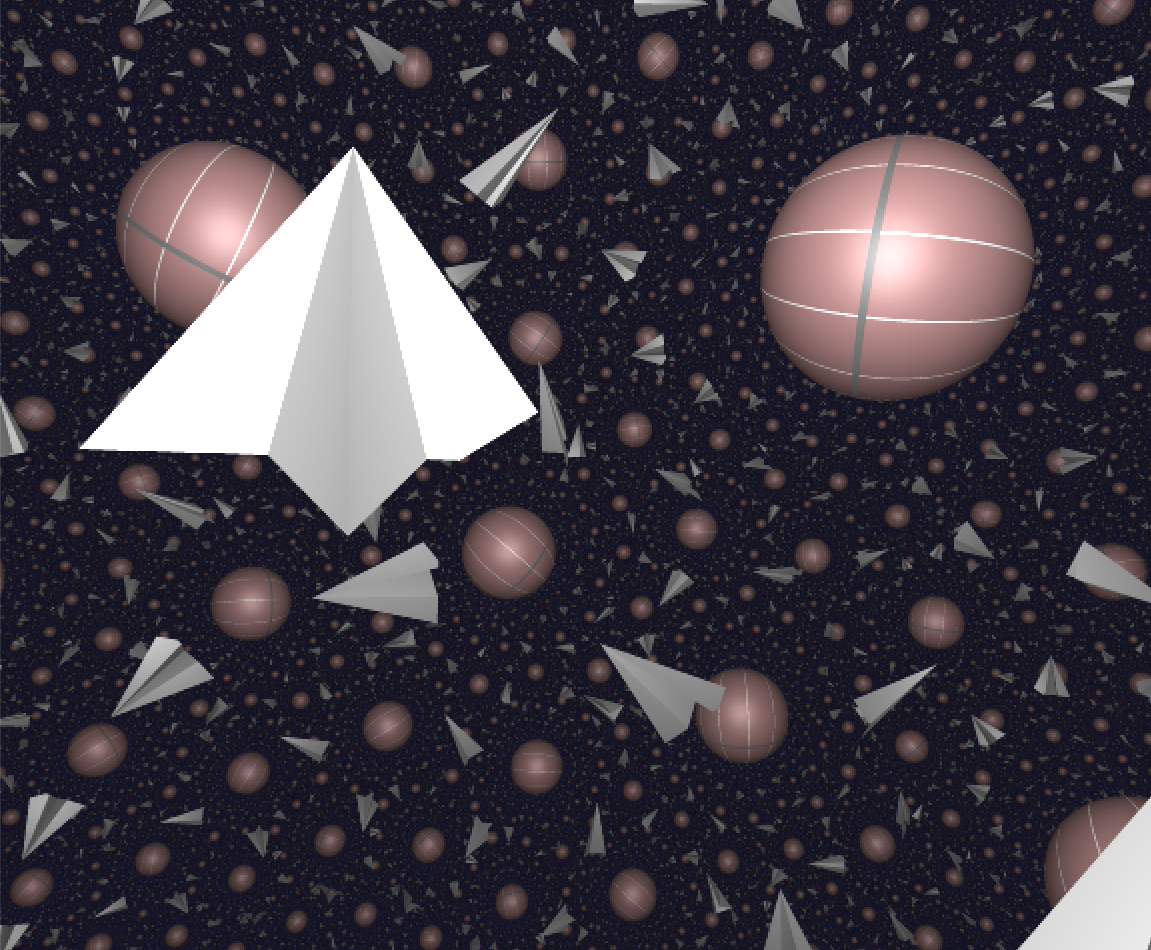}
\includegraphics[width=5cm]{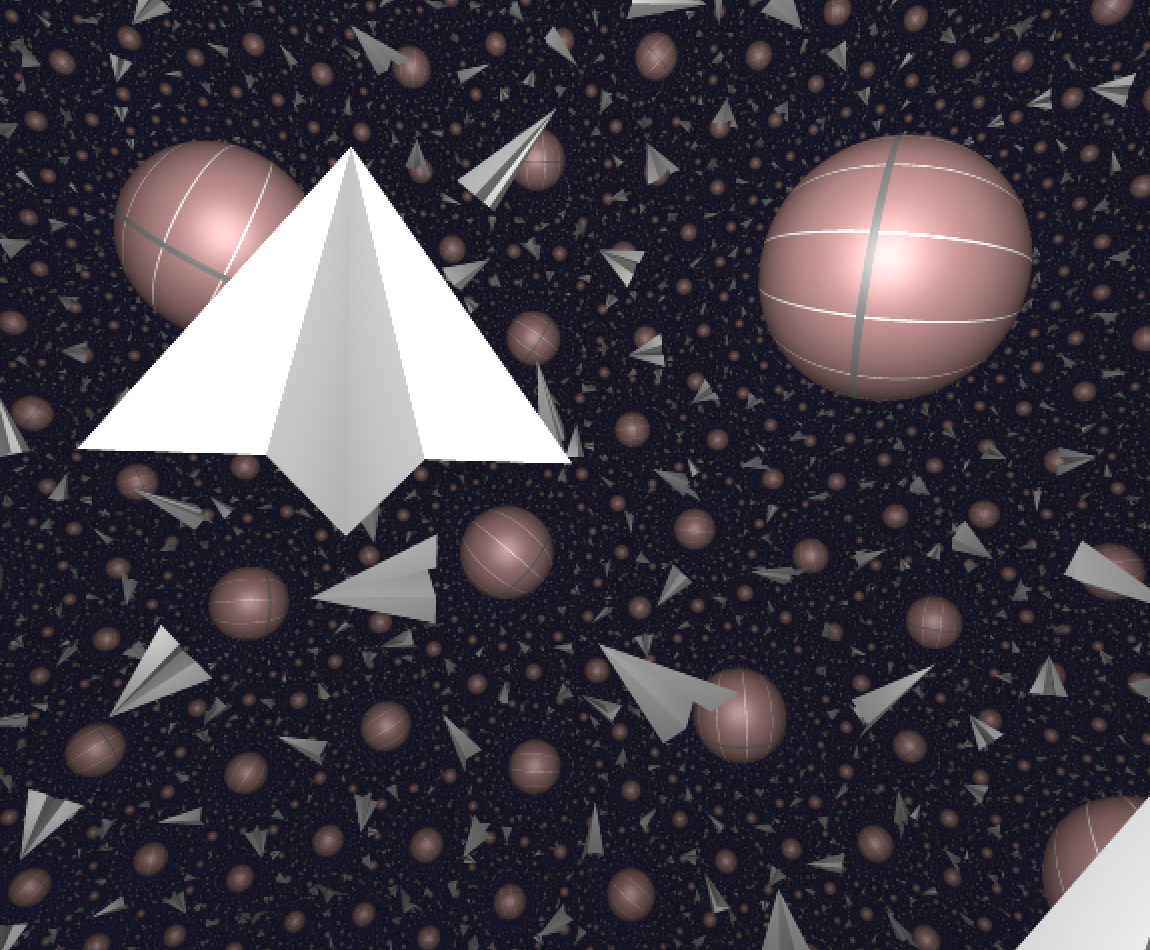}\\
\includegraphics[width=5cm]{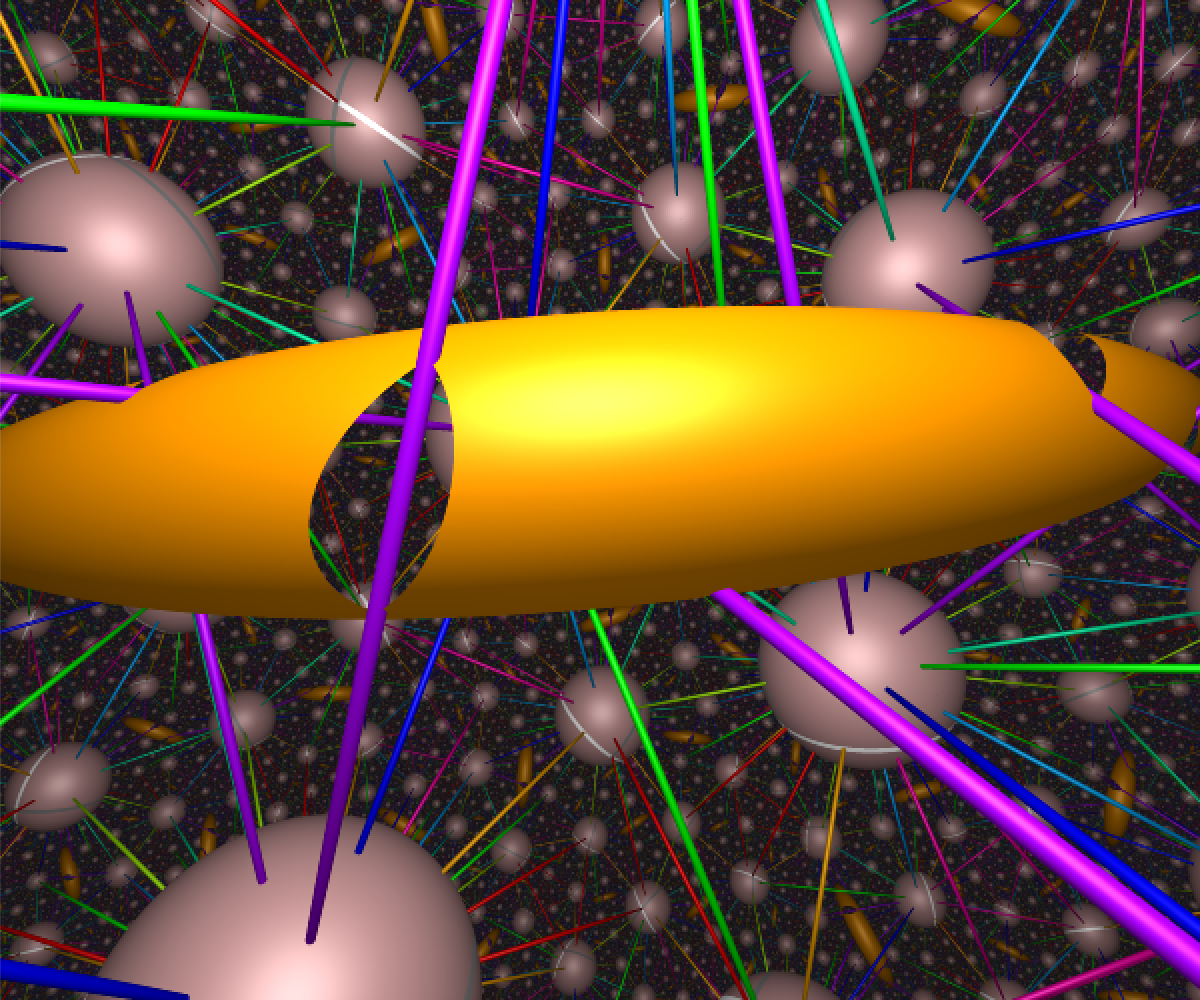}
\includegraphics[width=5cm]{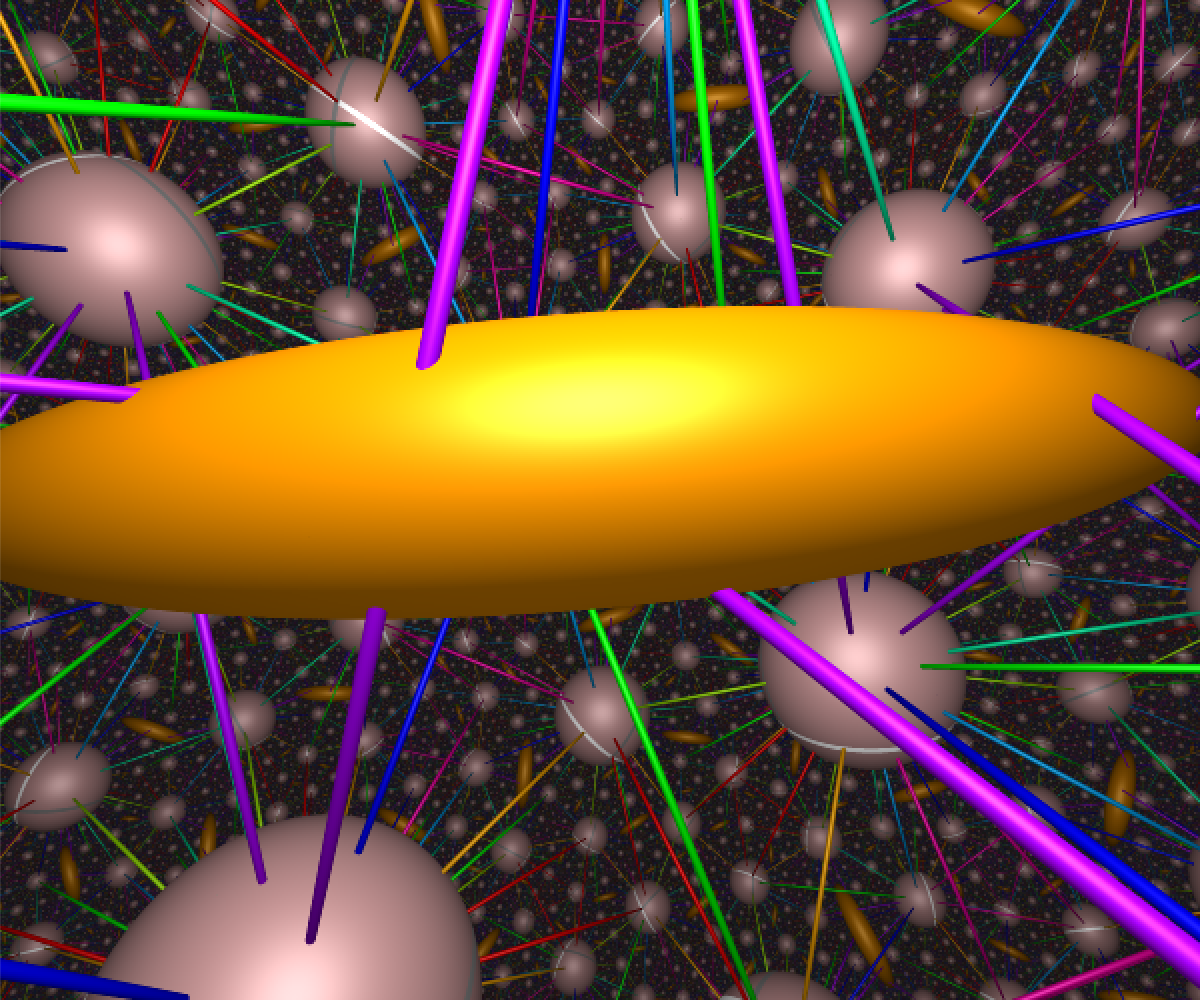}\\
\end{center}
\caption{Inside view with a paper plane in \texttt{m004} (top) and a tube about a geodesic (word \texttt{e}) in \texttt{o9\_34631} (bottom). We use the tiling algorithm to compute the corresponding lifted objects for the ray-tracing shader for the inside view. To demonstrate that this is necessary, we also show the picture when stopping the tiling algorithm prematurely (left) just before it places its last piece: the tip of the right wing of the paper plane is clipped (top) and the geodesic tube has a hole (bottom).
\label{fig:tilingRaytracing}}
\end{figure}

When we drop the last simplifying assumption, we no longer have $r_i\to\infty$ when $M$ is incomplete as the lifted (spun) tetrahedra approach a core curve without reaching it. More precisely, assuming we use the exact value for $d(K,mT_t)$, we have $r_i\to R$ where $R$ is the minimal distance of $K$ to any geodesic $\gamma\subset\H^3$ that is a lift of a core curve but not identical to $K$. This maximal tiling radius $R$ is given by:
\begin{itemize}
\setlength{\itemsep}{-0pt}
\item $R=\infty$ if $M$ is complete.
\item $R<\infty$ if $M$ is incomplete. \vspace{-4pt}
\begin{itemize}
\setlength{\itemsep}{-0pt}
\item $R$ is the minimal distance of $K$ in $M_\text{filled}$ to any core curve if $K$ is not a core curve. In particular, $R=0$ if $K$ is a geodesic intersecting a core curve transversally.
\item $R$ is the minimum of the minimal distance of $K$ in $M_\text{filled}$ to any other core curve and twice the embedding size of a tube about $K$ in $M_\text{filled}$ if $K$ is a core curve.
\end{itemize}
\end{itemize}
\begin{remark}
In \cite{ghht:lenSpec}, we show how to avoid $r_i$ being limited by the incompleteness locus $M_\text{filled}\setminus M$ by using truncated tetrahedra to tile around it.
\end{remark}

We still need to describe how to compute the distance $d(K, mT_t)$. It is given by the minimal distance $\min_{f=0,\dots,3} d(K, mT_t^f)$ to any face $f$ of $mT_t$ unless $K$ is a point inside $mT_t$ which only happens for the seeds for which we assume $r_i=-\infty$ anyway.

\subsection{The algorithm} \label{sec:tilingAlgoDetailSec}

We can now give the tiling algorithm: Algorithm~\ref{algorithm:tilingAny}. For the priority queue $Q$, we can use a binary heap as introduced in \cite{heapsort}. We describe the verified implementation of the set $S$ of lifted tetrahedra in $\Gamma_K\backslash \H^3$ in Sections~\ref{sec:verifiedDicts} and~\ref{sec:setLiftedTets}.

\begin{algorithm}[h]
\begin{tabular}{rp{15.4cm}}
{\bf Input:} & Developed fundamental polyhedron $P$ (see Section~\ref{sec:devFundPoly}) for hyperbolic 3-manifold $M$.\\
& Standard geometric object $K$ in $\H^3$ (see Definition~\ref{def:standardObject}).\\
& Lifted tetrahedra $m^\text{seed}_0T_{t^\text{seed}_0},\dots, m^\text{seed}_{k-1}T_{t^\text{seed}_{k-1}}$ such that at least one $m^\text{seed}_iT_{t^\text{seed}_i}$ is a seed for $K$ (see Definition~\ref{def:seed}).\\
{\bf Output:} & Stream $(r_0, m_0T_{t_0}), (r_1, m_1T_{t_1}), \dots$ with $m_iT_{t_i}$ lifted tetrahedra distinct in $\Gamma_K\backslash\H^3$ such that $m_0T_{t_0}, \dots, m_{i-1}T_{t_{i-1}}$ cover $\overline{B}_{r_i}(K)$ in $\Gamma_K\backslash \H^3$ for each $i$ and such that $r_i\to R$ (see Section~\ref{sec:tilingGeneralCase}).\\
{\bf Algorithm:}\\
\end{tabular}
\begin{algorithmSteps}
\item[Q] Priority queue of $(r, mT_t, f_\text{entry})$ with $r\in\{-\infty\}\cup\R$, $mT_t$ a lifted tetrahedron and $f_\text{entry}\in\{-1,0,\dots,3\}$. Use $\underline{r}$ as key.
\item[S] Set of lifted tetrahedra $mT_t$ in $\Gamma_K\backslash\H^3$ (see Sections~\ref{sec:verifiedDicts} and~\ref{sec:setLiftedTets}).
\item Add $(-\infty, m^\text{seed}_0T_{t^\text{seed}_0}, -1), \dots, (-\infty, m^\text{seed}_{k-1}T_{t^\text{seed}_{k-1}}, -1)$ to $Q$.
\item Repeat:
\begin{algorithmSteps}
\item Remove the $(r, mT_t, f)$ with the lowest key from $Q$.
\item If $mT_t$ is not in $S$:
\begin{algorithmSteps}
\item Emit $(r, mT_t)$.
\item Add $mT_t$ to $S$.
\item For each $f=0,\dots, 3$ with $f\not=f_\text{entry}$:
\begin{algorithmSteps}
\item Let $r'$ be the lower bound for $d\left(K,mT_t^{f}\right)$ from Section~\ref{sec:verifiedDistances}.
\item Add $(r', m'T_{t'}, f')$ to $Q$ where $t'=n_t^f$ and $f'=\sigma_t^f(f)$ and $m'=mg_{t'}^{f'}$.
\end{algorithmSteps}
\end{algorithmSteps}
\end{algorithmSteps}
\end{algorithmSteps}
\caption{The tiling algorithm. \label{algorithm:tilingAny}}
\end{algorithm}

To understand how the algorithm works, consider a boundary face of the tiled region; see Figure~\ref{fig:tilingExample}. Let $mT_t$ be the not yet enumerated lifted tetrahedron across that boundary face. Let $f_\text{entry}=0,\dots,3$ be the face of $mT_t$ corresponding to that boundary face. Let $r$ be the distance of that boundary face to $K$. This gives a triple $(r, mT_t, f_\text{entry})$. The priority queue $Q$ contains these triples for all the boundary faces (and potentially some other triples for already enumerated lifted tetrahedra). Thus, the head of $Q$ is a candidate for the next lifted tetrahedron to emit. We just need to check that this candidate was not enumerated already. 

\begin{remark}
We can easily modify graph tracing and the tiling algorithm to carry along with each matrix $m\in\Gamma$ the word in the face-pairing representation $\pi_1(P)$ associated with $P$.
\end{remark}

\begin{remark}
Note that $d\left(K, mT_t^f\right)=d\left(m^{-1}K,T^f_t\right)$ but computing $mT_t^f$ might be more expensive than $m^{-1}K$ and the lifted objects might be needed by the application anyway. Thus, we might consider rewriting the tiling algorithm in the tetrahedra view.
\end{remark}

\subsection{Verified computation} \label{sec:verifiedTiling}

We have the following guarantee for the verified version of Algorithm~\ref{algorithm:tilingAny}: There is a sequence $m_0T_{t_0},\allowbreak m_1T_{t_1},\dots$ of (exact) lifted tetrahedra (distinct in $\Gamma_K\backslash\H^3$) and a sequence of (exact) non-decreasing $r_0,r_1,\dots$ such that the $m_0T_{t_0}, \dots, m_{i-1}T_{t-1}$ cover $\overline{B}_{r_i}(K)\subset \Gamma_K\backslash\H^3$ for each $i$ and such that the intervals produced by the algorithm contain these values. In particular, the (exact) $m_0T_{t_0}, \dots, m_{i-1}T_{t-1}$ cover $\overline{B}_{\underline{r_i}}(K)\subset \Gamma_K\backslash\H^3$ where $\underline{r_i}$ is the left endpoint of the interval the algorithm produces for $r_i$.

While the true value of $r_i$ is non-decreasing and contained in the interval for $r_i$, the length of those intervals might fluctuate. Thus, $\underline{r_i}$ can fail to be non-decreasing. With fixed precision, the tiling algorithm could enter a mode where the intervals for $r_i$ grow faster than the true value of $r_i$ is increasing. This would result in an application waiting for a long time without making progress. However, in practice, such a scenario does not last for long as the growing intervals cause the algorithm to fail when looking up $mT_t$ in the set $S$.

By increasing precision, we can achieve any number of tiles to be produced before the algorithm fails. Similarly, we can make the maximum value of $\underline{r_i}$ to be arbitrary close to the maximal tiling radius $R$ (Section~\ref{sec:tilingGeneralCase}) before the algorithm fails. Future work might include adaptively increasing the precision of the intervals as we tile. It is unclear, however, how to exactly do this. For example, do we need to recompute all previous tiles when increasing the precision to tile further?

\section{Dictionaries for verified computations} \label{sec:verifiedDicts}

\subsection{Dictionary} \label{sec:dict}

A \emphasisText{dictionary} $D\from  X\rightharpoonup Y$ is a mutable partial assignment from keys in $X$ to values in $Y$. Typically, $D$ starts empty and has methods to set a value for a key or look-up a key. Here, we often use geometric objects such as matrices $m\in\Gamma$ as keys. For verified computations, we need dictionaries that work with interval estimates for a key in $X$. That is look-ups using different interval estimates for the same key need to return the same value and explicitly fail if the look-up is ambiguous. We introduce such dictionaries here.\label{sec:defHashDict}

\subsection{Hash dictionary}

Assume that we are given an equality predicate $\Eq$ that checks whether two given keys are the same given their interval estimates. A simple implementation $D_{\Eq}\from  X\rightharpoonup Y$ of a dictionary just stores all key-value pairs $(k,v)$ in an unstructured way. To look-up a key $k'$, we simply evaluate $\Eq(k,k')$ for each key-value pair $(k,v)$ in $D_{\Eq}$ and return the first match.

This, however, is not efficient. Thus, assume that we are also given a hash function $h\from  X\to\R$ for the keys. We call a dictionary $D_{\Eq,h}\from  X\rightharpoonup Y$ using $h$ to accelerate look-ups a \emphasisText{hash dictionary}.\label{sec:hashDict}

\begin{figure}
\begin{center}
\includegraphics[height=9cm]{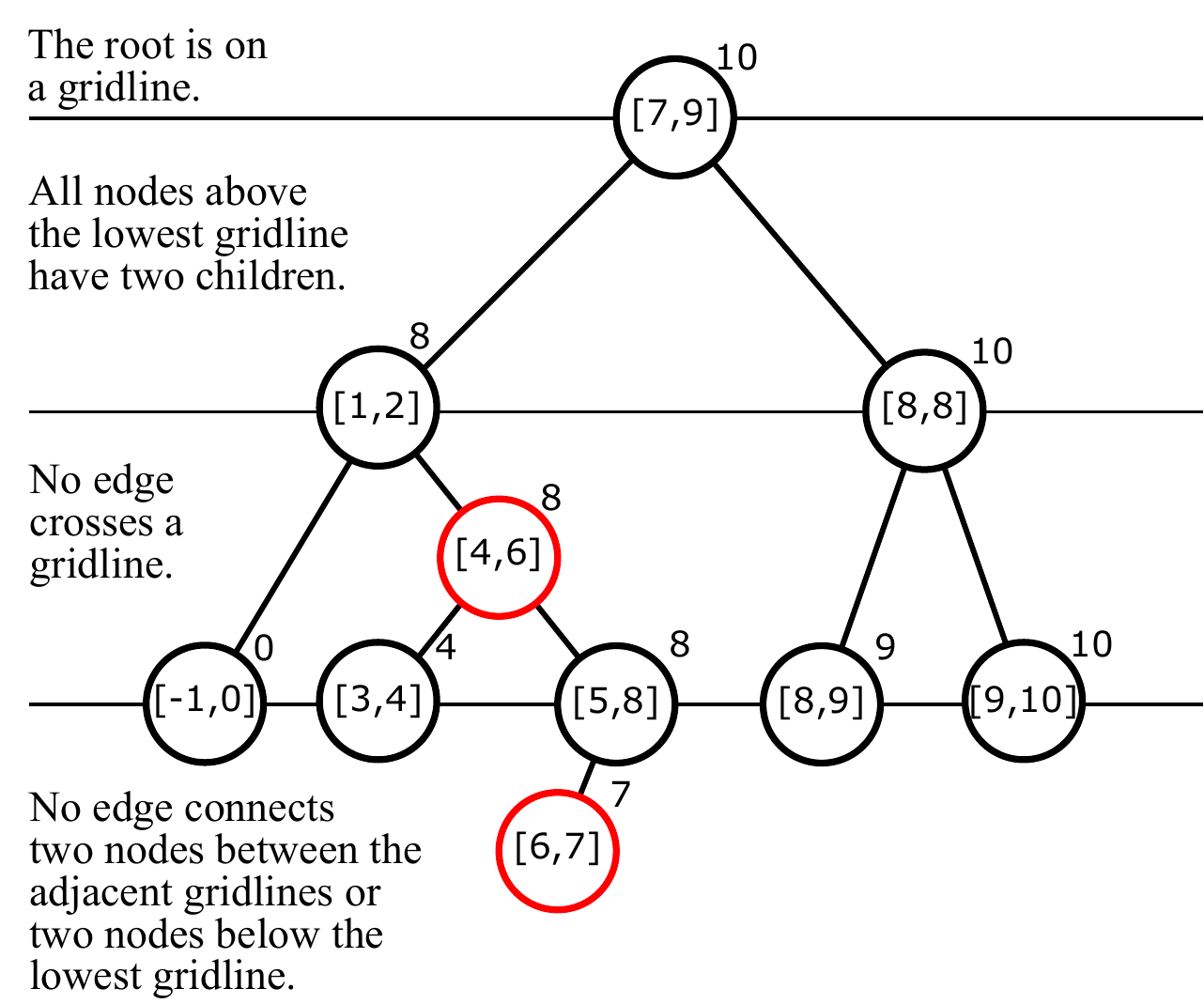}
\end{center}
\caption{An interval tree (introduced by Herbert Edelsbrunner) is a binary search tree sorted by the left endpoint of the interval. Each node is also annotated by the highest right endpoint among all its descendants including the node itself. The annotation helps to quickly dismiss irrelevant parts of the tree when looking for all intervals overlapping a given interval. We want to ensure that the binary search tree is balanced so that operations are $\mathcal{O}(\log n)$ where $n$ is the number of items. To do this, we require that it can be drawn on ruled paper with a lowest gridline such that the above rules are fulfilled. We leave it as exercise to the reader to figure out the necessary tree rotations (where a node takes the place of its parent and the parent becomes a child of the node) and vertical movements to reinstantiate the rules after inserting a new node as leaf. This is also known as red-black or dichromatic tree (introduced in \cite{redBlack}) where the color depends on whether a node is on a gridline or not.\label{fig:intervalTree}}
\end{figure}

For verified computation, we can implement a hash dictionary $D_{\Eq,h}$ using an associative data structure $D_\R$ for intervals. More precisely, $D_\R$ stores pairs $(I, v)$ where $I$ is an interval and $v$ a value. Given an interval $I'$, $D_{\R}$ allows quickly looking up us all pairs $(I, v)$ where $I$ and $I'$ overlap. The standard implementation of such a data structure\footnote{Usually, there are terms to distinguish an abstract data structure with a particular behavior (for example, a map) from an implementation (for example, a hash table). However, there seems to be no term for an associative data structure for intervals that is not alluding to its implementation as tree.} is an interval tree; see Figure~\ref{fig:intervalTree}.\label{sec:intervalDict}

To insert a new key-value pair $(k,v)$ into the dictionary $D_{\Eq,h}$, we add $(h(k), (k,v))$ to $D_\R$. To look-up a key $k'$ in $D_{\Eq,h}$, first look-up all key-value pairs $(I,(k,v))$ from $D_\R$ where $I$ overlaps $h(k')$. We now only need to evaluate the predicate $\Eq(k,k')$ for those pairs to find a match.

\subsection{Quotient dictionary} \label{sec:quotientDict}

A \emphasisText{quotient dictionary} $(X/\sim)\rightharpoonup Y$ can be constructed from a dictionary $X\rightharpoonup Y$ and a choice function $c\from  X\to X$ picking a canonical representative of the equivalence class of $X$. Note that intervals might not suffice to determine a unique canonical representative; see Section~\ref{sec:ambiChoice} for an example. Thus, we allow the choice function to return a set of candidates that has to include the ``true'' canonical representative (that is mathematically well-defined but hard to isolate among the candidates computationally).

Assume we look-up a key represented by some $x\in X$. In the underlying dictionary  $X\rightharpoonup Y$, we have to perform the look-up for each candidate returned by $c(x)$. Similarly when adding a key. If we want to be able modify the value of an existing key later, we store the value in the underlying dictionary by reference. This way, modifying the value of a key-value pair affects all key-value pairs in the same equivalence class.

\section{Set of lifted tetrahedra} \label{sec:setLiftedTets}

\subsection{Set of lifted tetrahedra}

Recall that the tiling algorithm (Algorithm~\ref{algorithm:tilingAny}) requires a set $S$ of lifted tetrahedra $mT_t$ in $\Gamma_K\backslash\H^3$. We can encode a lifted tetrahedron $mT_t$ as pair $(m, t)$ with $m\in\Gamma_K\backslash\Gamma$ and $t=0,\dots,{n-1}$. Thus, $S$ can be implemented as a subset of $(\Gamma_K\backslash\Gamma)\times\{0,\dots,n-1\}$.

We use that a subset of $A$ is equivalent to a dictionary $A\rightharpoonup Y$ where $Y$ is any one-element set. We also use currying to implement a dictionary $A\times B\rightharpoonup Y$ as $A\rightharpoonup B \rightharpoonup Y$. This reduces the problem to implementing a dictionary $\Gamma_K\backslash\Gamma\rightharpoonup Y$.

\subsection{Dictionary for discrete subsets of $\H^3 \cup \posLightCone$} \label{sec:hyperbolicDict} \label{sec:equalityPred}

Let $X$ be a subset of $\H^3 \cup \posLightCone$ such that there is a positive lower bound for $-x\cdot y$ with $x, y\in X$ and $x\not = y$. An example is $X=\Gamma p$ for a fixed $p\in\H^3$.

Let $b$ be such lower bound or an interval containing such a lower bound. We obtain an equality predicate $\Eq_b(x,y)$ for elements in $X$ suitable for verified computations as follows:
\begin{algorithmSteps}
\item If $-x\cdot y < b$: Return true.
\item If $-x\cdot y \geq b$: Return false.
\item Otherwise: Fail.
\end{algorithmSteps}
To avoid constant failure of $\Eq_b(x,y)$, the interval $b$ should not contain $0$ or the infimum of $-x\cdot y$.

Fix $\mu\in\R^{1,3}\setminus\{0\}$. Then a hash function is given by $\pi\from\H^3 \cup \posLightCone\to\R, x\mapsto x\cdot \mu$.

Let us denote by $D_{\H^3 \cup \posLightCone, b}$ the hash dictionary with equality predicate $\Eq_b(x,y)$ and hash function $\pi$; see Section~\ref{sec:defHashDict}.

\subsection{Dictionary $\Gamma\rightharpoonup Y$} \label{sec:gammaDict}

Let $\Gamma\subset\SO(1,3)$ be the discrete group associated with a developed fundamental polyhedron $P=\cup T_t$. Pick a tetrahedron $T_t$. Let $p\in\H^3$ and $r>0$ be the incenter and inradius of $T_t$. We convert the key using the map $\Gamma\to\H^3, g\mapsto gp$. We can then implement $D_\Gamma\from\Gamma\rightharpoonup Y$ using the dictionary $D_{\H^3 \cup \posLightCone, b}$ with $b=\cosh r$. 

\subsection{Dictionary $\Gamma_K\backslash\Gamma\rightharpoonup Y$}

\subsubsection{If $K$ is a point} \label{sec:casePoint}

Since $\Gamma_K$ is trivial, we can simply use the above dictionary $D_\Gamma\from\Gamma\rightharpoonup Y$ as dictionary $D_{\Gamma_K\backslash\Gamma}\from\Gamma_K\backslash\Gamma\rightharpoonup Y$.

\subsubsection{If $K$ is a horoball $B(l)$} \label{sec:dictHoroball}

To implement the dictionary $D_{\Gamma_K\backslash\Gamma}\from\Gamma_K\backslash\Gamma\rightharpoonup Y$, we again use $D_{\H^3\cup\posLightCone, b}$. This time, we convert a representative in $\Gamma$ of a key in $\Gamma_K\backslash\Gamma$ using the map $\Gamma\to\posLightCone, g\mapsto g^{-1}l$.

We compute a lower bound $b$ for $D_{\H^3 \cup \posLightCone,b}$ as follows. Consider the cusp neighborhood given by $B(l)$. Recall from Section~\ref{sec:addedCuspCrossSec} that $l$ is equal to a vertex $\vec{v}_t^f$ of a tetrahedron $T_t$. Reuse the cusp cross section used to scale the vertices for the corresponding cusp. Recall from Section~\ref{sec:stdFormArgs} the conditions for a cusp neighborhood to be embedded and intersect $\myTrig$ in standard form. Let $s$ be the largest scaling factor such that these conditions still apply. Set $b=s^2$. 

To see this suffices, we can equivalently set $b=1$ and scale $B(l)$ by $1/s$ such that the corresponding cusp neighborhood is embedded. Thus, if $B(l)$ and $B(l')$ are two different lifts of the cusp neighborhood, they are disjoint. This is equivalent to to $d(B(l),B(l'))\geq 0$ and thus $-l\cdot l' \geq 2>b$ by Section~\ref{sec:hypDistances}.

\begin{remark} \label{rem:bootstrap}
Finding $s$ can be regarded as a boot-strapping process to compute the maximal cusp area matrix $A_M$. That is, we use cusp neighborhoods in standard form to find some lower bound on the diagonal of the matrix. We use this to decide whether two tiles are distinct when tiling to compute the maximal cusp area matrix $A_M$.
\end{remark}

\subsubsection{If $K$ is a line $x_0x_1$}

Let $p\in\H^3$ and $b=\cosh r$ as in Section~\ref{sec:gammaDict}. We implement the dictionary $D_{\Gamma_K\backslash\Gamma}\from\discretionary{}{}{}\Gamma_K\backslash\Gamma\rightharpoonup Y$ as quotient dictionary (defined in Section~\ref{sec:quotientDict}) of $D_{\H^3 \cup \posLightCone, b}$. To convert the keys to $\H^3$, we again use the map $\Gamma\to\H^3, g\mapsto gp$. As choice function, we use $\pi_K\from \H^3\to\H^3$ defined as follows.

A fundamental domain of $\Gamma_K\backslash \H^3$ is bounded by two planes perpendicular to $K$ with one plane intersecting $K$ in the midpoint $(x_0+x_1)\postNorm$ and the other in a point closer to $x_1$. Let $h$ be the generator of $\Gamma_K\cong\Z$ with attractive fixed-point $x_1$. The choice function applies $h$ repeatedly to move a point into the fundamental domain.

For this, let $\lambda$ be the real length of the closed geodesic in $M$ corresponding to $K$. The function choice is given by
\[
\pi_K\from \H^3\to\H^3, x\mapsto h^{-\lfloor d(x)\rfloor} x\quad\mbox{where}\quad d(x)=\frac{d_{x_0x_1}\big((x_0+x_1)\postNorm, x\big)}{\lambda}.
\] \label{sec:ambiChoice}
We refer back to Section~\ref{sec:hypDistances} for evaluating $d(x)$. Given an interval estimate $[\underline{d}, \overline{d}]$ for $d(x)$, $\pi_K$ returns all $h^{-i} x$ where $i$ is an integer between $\lfloor \underline{d}\rfloor$ and $\lfloor \overline{d}\rfloor$.
\begin{remark}
We can also implement $D_{\Gamma_K\backslash\Gamma}$ using $D_{\H^3\cup\posLightCone, b}$ without using a quotient dictionary (similar to the case where $K$ is a horoball). For this, let $\pi_\text{Klein}\from \H^3\cup\posLightCone \to \R^{3,1}, (t,x,y,z)\mapsto (x/t,y/t,z/t,1)$ be the projection to the Klein model. We convert a representative of a key in $\Gamma_k\backslash\Gamma$ using the map
\[
\Gamma\to \H^3, m\mapsto \Big(\pi_{\partial B^3}\left(m^{-1}x_0\right) + \pi_{\partial B^3}\left(m^{-1}x_1\right)\Big)\postNorm.
\]
That is, we consider the endpoints of the transformed geodesic $m^{-1}K$ in the Klein model, take the Euclidean midpoint\oxford{} and project back to $\H^3$. We use $b=\cosh(r)$ where $r$ is the embedding size of a tube about the geodesic $K$ in $M$.\\
However, unlike in Remark~\ref{rem:bootstrap}, we cannot boot-strap this process since we compute such an embedding size using the tiling algorithm which needs the dictionary $D_{\Gamma_K\backslash\Gamma}\from\Gamma_K\backslash\Gamma\rightharpoonup Y$.
\end{remark}

\section{Computing distances in the manifold} \label{sec:embeddingSizes}

\subsection{Distance between standard geometric objects} \label{sec:distStandard}

Let $K$ and $K'$ be standard geometric objects for the hyperbolic 3-manifold $M$ such that $K$ and $K'$ correspond to the same cusp neighborhood in $M$ if they correspond to the same cusp of $M$. In this section, we discuss how to use the tiling algorithm to compute the distance $d_M(K,K')$ between these two objects in $M_\text{filled}$. Here, we again use the signed distance if $K$ or $K'$ is a horoball. If $K$ and $K'$ are the same in $M_\text{filled}$, then $d_M(K,K)$ is denoting twice the embedding size of $K$.

Figure~\ref{fig:tilingExample} shows an example of how to compute embedding size $d_M(x,x)$ of $x$ (or, equivalently, the infectivity radius about $x$).

More generally, $d_M(K,K')$ is given by the minimal distance between any lifts $m^{-1}K$ and $m'^{-1}K'$ with $m, m'\in\Gamma$ and $m^{-1}K$ and $m'^{-1}K'$ not the same as subsets of $\H^3$ (which is automatic if $K$ and $K'$ are different in $M$).

We introduce notation to handle multiple standard geometric objects simultaneously. Let $(r_0, m_0T_{t_0}),\allowbreak (r_1, m_1T_{t_1}), \dots, (r_{i-1}, m_{i-1}T_{t_{i-1}})$ be the first $i$ tiles emitted by the tiling algorithm (Algorithm~\ref{algorithm:tilingAny}) for the standard geometric object $K$. We denote the tiling radius and the lifted objects for the tetrahedron $T_t$ by
\[
r\big[K\big]_i=r_i\quad\mbox{and}\quad\Tiles\big[K\big]_{i,t}=\left\{ m_j^{-1} K\quad\text{where $j=0,\dots,i-1$ and $t_j=t$}\right\}.
\]
By definition, $\Tiles\big[K\big]_{i,t}$ includes all lifts $m^{-1}K$ whose neighborhood $\overline{B}_{r[K]_i}(m^{-1}K)$ intersects $T_t\subset\H^3$. For example, Figure~\ref{fig:tilingExample} shows $\Tiles\big[K\big]_{3,0}$ and $\Tiles\big[K\big]_{3,1}$. Here, we consider $\overline{B}_{r[K]_i}(m^{-1}K)$ empty unless $r\big[K\big]_i>d(K,\H^3)$. We have $d(K,\H^3)=0$ if the distance is unsigned and $d(K,\H^3)=-\infty$ otherwise (that is, if $K$ is a horoball). 

If $K$ and $K'$ are the same in $M$, we assume that we run the tiling algorithm only once for $K$, use its output for both $K$ and $K'$ and always have $i=i'$. We define:
\begin{align*}
d^t_{i,i'}(K, K') & = \min \left\{ d(A,A') \from  A\in \Tiles\big[K\big]_{i,t}~,~A'\in \Tiles\big[K'\big]_{i',t}~,~A\ne A'\right\}\\
d_{i,i'}(K, K') &= \min \left\{ d^t_{i,i'}(K, K') \from  t = 0, \dots, n-1\right\}
\end{align*}
Here, $A\not=A'$ is automatic if $K$ and $K'$ are distinct objects in $M$. Otherwise, $A\not=A'$ means that $A=m_j^{-1} K$ and $A'=m_{j'}^{-1} K$ come from elements with distinct indices $j$ and $j'$ in the stream of the tiling algorithm.

\begin{proposition} \label{prop:distEstimates}
Let $P$ be a developed fundamental polyhedron of a hyperbolic $3$-manifold $M$ that can be completed to $M_\text{filled}$ by attaching circles. Let $K$ and $K'$ be standard geometric objects. If $K$ and $K'$ correspond to the same cusp, we require that they are the same cusp neighborhood in $M$. If $K$ and $K'$ are the same in $M_\text{filled}$, we require that $i=i'$. Assume we use the tiling algorithm to compute $r\big[K\big]_i, r\big[K'\big]_{i'}$ and $d_{i,i'}(K,K')$ as above. If $r\big[K\big]_i>d(K,\H^3)$ and $r\big[K'\big]_{i'}>d(K',\H^3)$, we have for the distance $d_M(K,K')$ between $K$ and $K'$ in $M_\text{filled}$:
\[
\min\left\{d_{i,i'}(K,K'), r\big[K\big]_i + r\big[K'\big]_{i'}\right\} \leq d_M(K,K')\leq d_{i,i'}(K,K')
\]
In particular, if $r\big[K\big]_i + r\big[K'\big]_{i'} \leq d_{i,i'}(K,K')$, then \[ B_{r\big[K\big]_i}(K)\quad\text{and} \quad B_{r\big[K'\big]_i'}(K')\] are embedded (if $K$ and $K'$ are the same in $M$) or disjoint in $M$ (otherwise).\\
If $r\big[K\big]_i + r\big[K'\big]_{i'} \geq d_{i,i'}(K,K')$, then $d_M(K,K')=d_{i,i'}(K,K')$.
\end{proposition}

\begin{proof}
$d_M(K,K')$ and $d_{i,i'}(K,K')$ is the infimum across all lifts and across a subset of all lifts, respectively. Thus, $d_M(K,K')\leq d_{i,i'}(K,K')$.

If $r\big[K\big]_i + r\big[K'\big]_{i'}\leq d_M(K,K')$, $\min\left\{d_{i,i'}(K,K'), r\big[K\big]_i + r\big[K'\big]_{i'}\right\}\leq d_M(K,K')$ immediately follows. So let us assume that $r\big[K\big]_i+r\big[K'\big]_{i'}>d_M(K,K')$. Pick lifts $m^{-1}K$ and ${m'}^{-1}K'$ with $d(m^{-1}K, {m'}^{-1}K') = d_M(K,K')$. Since $d(m^{-1}K, {m'}^{-1}K') < r\big[K\big]_i+r\big[K'\big]_{i'}$, $B_{r[K]_i}(m^{-1}K)$ and $B_{r[K']_{i'}}({m'}^{-1}K')$ intersect. Let $p$ be a point in the intersection. Without loss of generality, we can assume that $p\in P$, say in tetrahedron $T_t$. Since $B_{r[K]_i}(m^{-1}K)$ intersects $T_t$, $m^{-1}K$ is contained in $\Tiles\big[K\big]_{i,t}$. Similarly for ${m'}^{-1}K'$ and $\Tiles\big[K'\big]_{i',t}$. Thus $d^t_{i,i'}(K,K')\leq d(m^{-1}K, {m'}^{-1}K')=d_M(K,K')$. By definition, $d_{i,i'}(K,K')\leq d^t_{i,i'}(K,K')$. Thus, we have $d_{i,i'}(K,K')\leq d_M(K,K')$.
\end{proof}

Note that $d_{i,i'}(K,K')$ is an upper bound for $d_M(K,K')$. Recall from Section~\ref{sec:intConv} that, thus, $d_M(K,K')\leq\overline{I}$ where $I$ is the interval computed for $d_{i,i'}(K,K')$. Analogously, for the lower bound for $d_M(K,K')$.

We can now compute $d_M(K,K')$ as follows: First take enough tiles from the tiling algorithms for $K$ and $K'$ to ensure $r\big[K\big]_i>d(K,\H^3)$ and $r\big[K\big]_{i'}>d(K',\H^3)$, respectively. Then continue taking tiles until $r\big[K\big]_i+r\big[K'\big]_{i'} > d_{i,i'}(K, K')$. Now $d_M(K,K')$ is given by $d_{i,i'}(K,K')$. When using intervals, this means that $d_M(K,K')$ is contained in the interval computed for $d_{i,i'}(K,K')$.

If $K$ and $K'$ are distinct in $M_\text{filled}$, we can take tiles from for either just $K$ or just $K'$ or alternate at each step to not give preference to one or the other. The latter seems to be performing best.

\subsection{The maximal cusp area matrix} \label{sec:compMaxCuspAreaMatrix}

Let $M$ be a complete hyperbolic 3-manifold with ideal geometric triangulation $\myTrig$. Pick a cusp cross section for each cusp to scale the vertices $\vec{v}_t^v$ of the developed fundamental polyhedron $P$ as described in Sections~\ref{sec:cuspCrossSec} and~\ref{sec:addedCuspCrossSec}. Let $K$ and $K'$ be a vertex $\vec{v}_t^v$ of $P$ corresponding to cusp $i$ and $j$, respectively. Let $A(C_i)$ and $A(C_j)$ be the area of the chosen cusp cross section for cusp $i$ and $j$, respectively. Then, the corresponding entry in the maximal cusp area matrix is given by:
\[
A_{ij}= e^{2d_M(K,K')} A(C_i)A(C_j)
\]

\subsection{Lower bound for disjoint and embedded neighborhoods} \label{sec:lowerBoundSystem}

As an immediate consequence of Proposition~\ref{prop:distEstimates}, we can also obtain a lower bound $r$ for neighborhoods of a system of standard objects to be embedded and disjoint. We use it in \cite{goerner:drilling} to compute a safe perturbation distance for a system of geodesics. Here, safe means that perturbing up that distance does not change the isotopy type of the system.

\begin{corollary}
Let $K_0,\dots, K_{s-1}$ be standard geometric objects corresponding to different objects in the hyperbolic 3-manifold $M$. Let us run the tiling algorithm for each $K_j$ for $i_j$ steps until $r\big[K_j\big]_{i_j} > d(K_j, \H^3)$. Let
\[
d^t=\min\left\{d(A,A'): A,A' \in \Tiles\big[K_0\big]_{i_0,t} \cup \dots \cup\Tiles\big[K_{s-1}\big]_{i_{s-1},t} , A\ne B\right\}
\]
be the minimum distance between any two distinct lifts $m^{-1}K_j$ the tiling algorithms have produced for the tetrahedron $T_t$. Then, the following gives a lower bound $r$ for the neighborhoods $B_r(K_j)$ to be embedded and disjoint in M:
\[
r=\frac{1}{2}\min\left\{d^0,\dots,d^{n-1},r\big[K_0\big]_{i_0},\dots,r\big[K_{s-1}\big]_{i_{s-1}} \right\}
\]
\end{corollary}

\begin{proof}
$d^t$ is the minimum of all $d^t_{i_j,i_{j'}}(K_j,K_{j'})$ with $j,j'=0,\dots,s-1$. We use that the minimum of $r\big[K_j\big]_{i_j}$ and $r\big[K_{j'}\big]_{i_{j'}}$ is smaller than their average.
\end{proof}

\section{Cusp areas and 6-Theorem} \label{sec:revisitSix}

\subsection{Cusp areas}

In this section, we discuss how to find the areas of maximal embedded and disjoint cusp neighborhoods from the maximal cusp area matrix $A_M$ (see Definition~\ref{def:maxCuspAreaMatrix}). Different choices of such neighborhoods are possible if there are multiple cusps. These choices include:
\begin{itemize}
\item The \emphasisText{greedy choice}: grow a preferred neighborhood until it bumps into itself; then grow the neighborhood about the second preferred cusp until it bumps either into the first neighborhood or into itself; and so on.
\item The \emphasisText{unbiased choice}: grow all neighborhoods simultaneously with each neighborhood stopping when it bumps into itself or another neighborhood.
\end{itemize}

\begin{algorithm}[ht]
\begin{tabular}{rp{15.4cm}}
{\bf Input:} & Symmetric $n\times n$-matrix $A=(A_{ij})$ with $A_{ij}\in\R^+$.\\
{\bf Output:} & Unbiased maximal vector $a=(a_0,\dots,a_{n-1})$ with $a_i\in\R^+$ such that $aa^T\leq A$ (element-wise).\\
{\bf Algorithm:}
\end{tabular}
\begin{algorithmSteps}
\item[a] $n$-vector $(a_0,\dots,a_{n-1})$ with $a_i\in\R^+\cup\{t\}$ where $t$ is a formal variable.
\item If not $A_{ij}>0$ for all $i, j=0,\dots,n-1$: Fail.
\item Initialize $a=(t,\dots, t)$.
\item While $a$ has a $t$:
\begin{algorithmSteps}
\item For each $i'=0,\dots,n-1$ with $a_{i'}=t$ and for each $j'=0,\dots,n-1$:
\begin{algorithmSteps}
\item $t_{i'j'}\leftarrow
\begin{cases}
\sqrt{A_{i'j'}} & \text{if}~a_{j'}=t\\
A_{i'j'}/a_{j'} & \text{otherwise}.\\
\end{cases}
$
\end{algorithmSteps}
\item Let $\underline{t_\text{next}}$ be the value of and $i$ and $j$ be the index of the smallest $\underline{t_{i'j'}}$ (pick any if tied).
\item Let $\overline{t_\text{next}}$ be the smallest among the upper bounds for
$$
\begin{cases}
\sqrt{A_{ii}} & \text{if}~j'=i\\
A_{ij'}/\underline{t_\text{next}} & \text{if}~j'\ne i~\text{and}~a_{j'}=t\\
A_{ij'}/a_{j'} & \text{otherwise}\\
\end{cases}\qquad\text{where}~j'=0,\dots,n-1.
$$
\item $a_i\leftarrow \left[\underline{t_\text{next}},\overline{t_\text{next}}\right]$
~~~~~[Note: $a_i \leftarrow \underline{t_\text{next}}$ when using exact arithmetic]
\end{algorithmSteps}
\item Return $a$.
\end{algorithmSteps}
\caption{Verified unbiased maximal cusp areas from cusp area matrix. \label{algo:unbiasedChoice}}
\end{algorithm}

An interval version of the unbiased choice is implemented by Algorithm~\ref{algo:unbiasedChoice}. By the following proposition, the result is correct for intervals even though the intervals for the input matrix $A$ might not completely determine the order in which neighborhoods bump into each other and the algorithm might fill the vector $a$ in an order different from the true order:

\newpage

\begin{proposition}
Algorithm~\ref{algo:unbiasedChoice} fulfills the inclusion principle (see Definition~\ref{def:inclusionPrinciple}).
\end{proposition}

\begin{proof}
Fix a symmetric matrix of positive real numbers for $A_{ij}$. We can think of such a real number as the ``true'' value of each $A_{ij}$. Execute the algorithm on this matrix with exact arithmetic to obtain the ``true'' value for each $a_i$. Consider intervals for each $A_{ij}$ that each contains the true value of $A_{ij}$. We need to show that, when given the intervals for $A_{ij}$, the intervals returned by the algorithm contain the true value of each $a_i$. We use induction on the iterations: assume that each interval entry in $a$ contains its true value, we need to show that the interval $\left[\underline{t_\text{next}},\overline{t_\text{next}}\right]$ assigned to $a_i$ contains its true value. Note that $\underline{t_\text{next}}$ is a lower bound for the true value of all entries in $a$ that have not been assigned an interval yet. For each $j'=0,\dots,n-1$, the inequality $a_ia_{j'}\leq A_{ij'}$ for the true values gives us an upper bound for $a_i$ as follows: If $j'=i$, then we know that $a_ia_i\leq A_{ii}$ for the true value, so an upper bound for $\sqrt{A_{ii}}$ is an upper bound for $a_i$. Otherwise, we obtain an upper bound for $a_i$ by dividing an upper bound for $A_{ij'}$ by a lower bound for $a_{j'}$. This lower bound is given by the left-endpoint of the interval for $a_{j'}$ if already computed earlier and $\underline{t_\text{next}}$ otherwise. $\overline{t_\text{next}}$ is the best among all these upper bounds.
\end{proof}

\subsection{Conventions for peripheral curves} \label{sec:peripheralConventions}

Let $M$ be a finite-volume, complete, orientable hyperbolic $3$-manifold. Following SnapPy conventions, we pick two peripheral curves called meridian and longitude spanning a cusp's homology for each cusp $i = 0,\dots, n-1$. 
To Dehn-fill one or multiple cusps, we give pairs $(p_i, q_i)$ of co-prime integers that give a simple peripheral curve in terms of meridian and longitude. The cusp translations $\mu_i\in\C$ and $\lambda_i\in\C$ associated with the meridian and longitude, respectively, depend on choosing a cusp neighborhood for cusp $i$. However, the cusp shape $s_i=\lambda_i/\mu_i$ is well-defined without such a choice.

\begin{example}
SnapPy can compute verified cusp shapes as follows:
\texttt{Manifold(\discretionary{}{}{}"m125").cusp\_info(\discretionary{}{}{}"shape", verified=True)}.
\end{example}

\subsection{Conditions of the 6-Theorem} \label{sec:conditionsSix}

The length of a peripheral curve measured along the boundary of a cusp neighborhood $C_i$ is given by:
\[
l_i = \sqrt{\frac{A(C_i)}{\ImPart(s_i)}} \left| p_i + q_i s_i \right| .
\]

If the cusp neighborhoods $C_i$ are disjoint and $l_i > 6$ for each cusp we fill, the Dehn-filled manifold is hyperbolic by the $6$-theorem (see \cite{Agol:Six,Lackenby:Six}) and the geometrization theorem by Perelman (see \cite{perel1,perel2}).

When classifying exceptional slopes, for example, one typically picks disjoint and embedded cusp neighborhoods first and then lists the slopes with $l_i \leq 6$.

Going in the other directions, assume we are given Dehn-fillings $(p_i,q_i)$. Compute the area $a_i$ of cusp $C_i$ such that $l_i = 6$:
\[
a_i= \begin{cases}
\ImPart(s_i) \left(\frac{6}{\left| p_i + q_i s_i\right|}\right)^2 & \text{if } (p_i,q_i)\not=(0,0) \\ 
0 & \text{otherwise.}
\end{cases}
\]
Let $a=(a_i)$ and let $A_M$ be the maximal cusp area matrix. Then, there are cusp neighborhoods such that the conditions of the $6$-Theorem apply if and only if $aa^T < A_M$ (element-wise).

In particular, all Dehn-fillings of $M$ are hyperbolic if $\alpha\alpha^T < A_M$ where $\alpha=\left(\alpha_i\right)$ and $\alpha_i$ is the maximum value of $a_i$ among all co-prime $(p_i, q_i)$.

\section{Epstein-Penner decomposition} \label{sec:epAlgorithm}

\subsection{Definition}

Let $M$ be a complete hyperbolic 3-manifold. Let $\Gamma\subset\SO(1,3)$ such that $M\cong\Gamma\backslash\H^3$. Furthermore, pick a cusp neighborhood $C_i$ for each cusp. Each cusp neighborhood lifts to horoballs $B(l)$ with $l\in\posLightCone$. Let $V\subset\posLightCone$ be the set of all $l$. Let $C$ be the convex hull of $V$ in $\R^{1,3}$. Let $Q=\partial C$ be its boundary.
Project $Q$ to $\H^3$. The Epstein-Penner decomposition (see \cite{EP}) is then defined as the quotient of the induced polyhedral decomposition of $\H^3$ by $\Gamma$. The Epstein-Penner decomposition is invariant under scaling all cusp neighborhoods by the same factor.

\subsection{Algorithm}

Algorithm~\ref{algorithm:epAlgorithm} uses the tiling algorithm to compute the Epstein-Penner decomposition. We use an exact type to represent numbers which are all algebraic as long as the input areas $A(C_i)$ are. The author has not implemented this algorithm and it is unclear how feasible it is in practice.
\begin{algorithm}[ht]
\begin{tabular}{rp{15.4cm}}
{\bf Input:} & Area $A(C_i)$ for each cusp $i$ of complete hyperbolic 3-manifold $M$.\\
& Developed fundamental polyhedron $P$ containing origin  $x=(1,0,0,0)$ for $M$.\\
{\bf Output:} & Epstein-Penner decomposition.\\
{\bf Algorithm:}
\end{tabular}
\begin{algorithmSteps}
\item Scale vertices $\vec{v}_t^v$ of $P$ such that $B\left(\vec{v}_t^v\right)$ corresponds to a cusp neighborhood of area $A(C_i)$ if vertex $v$ of $T_t$ corresponds to cusp $i$. \label{step:scaleVert}
\item Scale vertices $\vec{v}_t^v$ by the largest overall factor such that each $B\left(\vec{v}_t^v\right)$ intersects $T_t$ in standard form. \label{step:scaleToStd}
\item With an increasing number of tiles $(r_0, m_0T_{t_0}), \dots, (r_{k-1}, m_{k-1}T_{t_{k-1}})$ from Algorithm~\ref{algorithm:tilingAny} where $K=x$, try the following: \label{step:tileEP}
\begin{algorithmSteps}
\item If not $r_k>0$: Continue to next iteration.
\item Let $V_k\subset\posLightCone$ be the set of all vertices $m_i\vec{v}_{t_i}^v$ (requires de-duplication). Let $C$ be the convex hull of $V_k$ in $\R^{1,3}$ (computed using Quickhull \cite{quickhull}). Let $Q\subset\partial C$ be the subset of faces ignoring faces with outward-facing normal pointing up. Let $\pi\from  Q\to \H^3$ be the projection onto $\H^3$.
\item For each $3$-cell $U$ in $Q$ with  $\pi(U)$ intersecting $P$, consider the hyperplane $H$ in $\R^{1,3}$ supporting $U$. Let $(t,x,y,z)$ be the point in $H\cap \posLightCone$ with the largest $t$. Check that $t \leq e^{r_k}$. If any check fails: Continue to next iteration.
\item For each $3$-cell $U$ in $Q$ with $\pi(U)$ intersecting a boundary face $T^f_t$ of $P$, check that $g^f_tU$ is equal to a $3$-cell $U'$ of $Q$ and identify $U$ with $U'$. If any check fails: Continue to next iteration.
\item Return the cell decomposition obtained by all $3$-cells $U$ in $Q$ with $\pi(U)$ intersecting $P$ after the above identifications.
\end{algorithmSteps}
\end{algorithmSteps}
\caption{Epstein-Penner decomposition.\label{algorithm:epAlgorithm}}
\end{algorithm}

To see that the algorithm is correct, we need to prove that the algorithm has enumerated enough elements of $V$. That is, elements of $V\setminus V_k$ change $Q$ in regions not intersecting the fundamental polyhedron $P$. This follows from the following proposition:
\begin{proposition}
Let $l=(t,x,y,z)\in V$. If $B(l)$ intersects $B_{r_k}((1,0,0,0))$, or, equivalently $t< e^{r_k}$, then $l\in V_k$ where $V_k$ is as defined in Algorithm~\ref{algorithm:epAlgorithm}.
\end{proposition}
\begin{proof}
$d((1,0,0,0), B((t,x,y,z))) = \log(t) < 0$ if and only if $B(l)$ intersects $B_{r_k}((1,0,0,0))$.
Take a point $p$ in $B(l)\cap B_{r_k}((1,0,0,0))$. Since $B(l)$ intersects all lifted tetrahedra in standard form, there is some lifted tetrahedron $mT_t$ containing $p$ and having $l$ as a vertex. Thus, $mT_t$ is among the $m_0T_{t_0}, \dots, m_{k-1}T_{t_{k-1}}$. Thus, $l\in V_k$.
\end{proof}

\section{Generalization to higher dimensions} \label{sec:higherDim}

In this section, we generalize geometric triangulations and the tiling algorithm in Section~\ref{sec:tiling} to higher dimensions. The goal is to use the algorithm in Section~\ref{sec:epAlgorithm} to compute the Epstein-Penner decomposition in any dimension. Thus, we focus on ideal triangulations of cusped manifolds and the case where the object $K$ we tile about is a point.

We follow the approach of \cite[Section~1]{heardThesis} and \cite[Section~2 and~3]{matthiasVerifyingFinite} and use the same notation here. Some of the ideas go back to Casson's Geo \cite{casson:Geo} and Frigerio, Martelli\oxford{} and Petronio's ographs \cite{FMP:ographs}.

\subsection{Natural cocycle for a vertex Gram matrix} \label{sec:natCocycleGram}

We generalize \cite[Section~1]{heardThesis} and \cite[Section~2 and~3]{matthiasVerifyingFinite} to higher dimensions and apply it to ideal $n$-simplices and as follows. In \cite[Section~3]{matthiasVerifyingFinite}, $\overline{\Delta}\subset\Delta$ is replaced by $\overline{\Delta}^n\subset\Delta^n$, the $n-1$-times truncated simplex which is the permutahedron of order $n+1$. Each vertex of $\overline{\Delta}$ corresponds to a permutation $\sigma\in S^{n+1}$. We label an oriented edge $e$ of $\overline{\Delta}$ by $(\sigma, k)$ if $e$ starts at $\sigma$ and ends at $\sigma \circ (k, k+1)$. We use the hyperboloid model $\H^n\subset\R^{1,n}$ with isometry group $\SO(1,n)$. We decompose $\R^{1,n}$ as $\R^{1,n}_-\cup \R^{1,n-1}\cup\R^{1,n}_+$ where $\R^{1,n}_-=\{x\in\R^{1,n}:x_n <0 \}$, $\R^{1,n-1}\subset\R^{1,n}$ and $\R^{1,n}_+=\{x\in\R^{1,n}:x_n >0\}$. Similar to \cite[Definition~3.2]{matthiasVerifyingFinite}, we say that a positively oriented $n$-simplex $\Delta^n\subset\H^n$ with ideal vertices $\vec{v}^v\in\mathbb{L}_{+}^n$ (and thus a choice of horoballs about its vertices) is in \emphasisText{$\sigma$-standard position} if $\vec{v}^{\sigma(0)}=(1,-1,0,\dots)$, $\vec{v}^{\sigma(k)}\in\R_+^{1,k}$ for $k=1,\dots,n-1$ and $\vec{v}^{\sigma(n)}\in\R_{\sgn(\sigma)}^{1,n}$. Given an isometry class of such simplices with choices of horoballs, we define the \emphasisText{natural $\SO(1,n)$-cocycle} on $\overline{\Delta}^n$ analogously to \cite[Definition~3.3]{matthiasVerifyingFinite}.

Following \cite[Section~1.3]{heardThesis}, the isometry class of a geodesic $n-$simplex (including the choice of horoballs if ideal) is determined by the vertex Gram matrix $G$ which has entries $\vec{v}^i\cdot \vec{v}^j$. The vertex Gram matrix $G$ also determines the normal Gram matrix $G^*=-\det(G) G^{-1}$. Following \cite[Section~1.3]{heardThesis}, we can compute the dihedral angle $\theta_{i,j}$ between face $i$ and $j$ from $G^*$.

We now compute the natural cocycle $\alpha$ from $G$ for an ideal simplex. We denote by $\alpha^n_{\sigma, k}$ the matrix that $\alpha$ assigns to the edge $(\sigma, k)$ of $\overline{\Delta}^n$. Let $R_\theta\in\SO(1,n)$ be the rotation by angle $\theta$ fixing $\H^{n-2}\subset\H^n$ point-wise such that $\R_+^{1,n-1}$ is taken to $\R_+^{1,n}$ for small $\theta>0$. By symmetry, it is sufficient to give an expression for each $\alpha^n_{\Id,k}$ with $k=0,\dots,n-1$. They are given by:
\begin{itemize}
\item $\alpha^n_{\Id, n-1}=R_{\theta_{n-1,n}}$.
\item $\alpha^n_{\Id, k}=R_\pi \circ i(\alpha^{n-1}_{\Id, k})$ for $k=0,\dots, n-1$.\\
Here, $i\from \SO(1,n-1)\to\SO(1,n)$ is the inclusion and $\alpha^{n-1}_{\Id,k}$ is computed from the vertex Gram matrix obtained by deleting the last row and column of $G$. That is, we reduce the case to one dimension lower by regrading face $n$ of $\Delta^{n}$ as $n-1$-simplex $\Delta^{n-1}$. This eventually reduces it to the case $\overline{\Delta}^2$.
\item $\alpha^2_{\Id, 1}$ is the $\SO(1,2)$-matrix acting on the horocycle $\partial B(\vec{v}^0)=\partial B((1,-1,0)$ as Euclidean translation by distance $d=d_{\partial B(\vec{v}^0)}(\vec{v}^1,\vec{v}^2)$. We can evaluate $d$ using the entries of $G$ and the respective expression in Section~\ref{sec:hypDistances}.
\item $\alpha^2_{\Id, 0}\in \SO(1,2)$-matrix is the swapping $\vec{v}^0=(1,-1,0)$ with the ideal vertex $\vec{v}^1\in\R_+^{1,1}$ determined by the off-diagonal entry $\vec{v}^0\cdot\vec{v}^1$ of $G$.
\end{itemize}

\subsection{Hyperbolic structure for a triangulation} \label{sec:hypStructHigher}

Let $\myTrig=\cup T_t$ be an oriented $n$-triangulation. We assume that no face of $\myTrig$ is identified with itself under a non-identity permutation. This is a necessary condition for the geometric realization of $\myTrig$ to be a manifold.
\begin{example}
This condition is violated if an edge is identified with itself in reverse in a (necessarily non-orientable) $3$-triangulation. The link of the corresponding edge midpoint in the geometric realization is a projective plane.
\end{example}

To each simplex $T_t$, assign a vertex Gram matrix $G_t$ corresponding to a generalized hyperbolic simplex. Such vertex Gram matrices are characterized by a generalized version of \cite[Theorem~1.5]{heardThesis}. Assume that the $G_t$ are compatible, that is that the values assigned to the vertices and edges of each $T_t$ match under the gluings.

Let $\pentagon_d$ denote a $d$-gon. Analogously to \cite[Section~3]{matthiasVerifyingFinite}, $\hat{\myTrig}$ is obtained from $\overline{\myTrig}$ by adding a prism $\pentagon_d\times \overline{\Delta}^{n-2}$ for each $n-2$-cell $e$ of $\myTrig$ where $d$ is the degree of $e$.
 
 Assume the cocycle on $\overline{\myTrig}$ extends to $\hat{\myTrig}$. This is equivalent to the cocycle condition holding for a $\pentagon_d$ of each prism. Note that this condition is algebraic in the entries of the $G_t$. It is also equivalent to each $\Theta_e$ being a multiple of $2\pi$ where $\Theta_e$ is the sum of dihedral angles adjacent to $e$.
 
The inclusion $\hat{\myTrig}\subset\myTrig\setminus\myTrig_0$ induces a fully faithful functor between the corresponding fundamental groupoids. Thus, we obtain a representation $\rho\from\pi_1(\myTrig\setminus\myTrig_0)\to\SO(1,n)$ and can use Section~\ref{sec:cocycleDev} to obtain a vertex developing map $\Psi\from \tilde{V}\to\H^n$ using $q=(-1,1,0,\dots,0)$.

Assume that each $G_t$ corresponds to an ideal simplex. The $G_t$ give us incomplete a hyperbolic manifold structure on the open $n$-cells. We show how to inductively verify whether this hyperbolic manifold structure extends to the open $n-1$-cells, open $n-2$-cells, \dots, open $1$-cells so that we eventually have a complete hyperbolic manifold structure on $\myTrig\setminus\myTrig_0$.

\begin{remark}
If the $G_t$ correspond to finite simplices, we can extend this process to the $0$-cells to obtain a hyperbolic manifold structure on $\myTrig$. We can also generalize to the hyperideal case.
\end{remark}

Assume that we have verified that the hyperbolic structure extends down to the open $k+1$ cells and we now need to verify that it extends to the $k$-cells. We distinguish between the cases $k=n-1$, $k=n-2$ and $n-2 > k > 0$.

The case $k=n-1$ follows from the compatibility conditions on the $G_t$.

For the case $k=n-2$, we need to verify that each $\Theta_e=2\pi$.

For the other cases, note that the link $L$ of a $k$-cell $c$ has a triangulated spherical $n-k-1$-manifold structure. We need to verify that $L$ is $S^{n-k-1}$ rather than a quotient of $S^{n-k-1}$ for each $k$-cell $c$. This is equivalent to showing that the corresponding representation $\pi_1(L)\to\SO(n-k)$ is trivial.

\begin{example}
A non-orientable 3-triangulation can have a projective plane as vertex link $L$. We obtain a representation $\pi_1(L)\to\myO(3)$ sending the generator of the projective plane to $-\Id$ if we amend the theory to non-orientable triangulations. That is, we change the standard-position to $\vec{v}^{\sigma(n)}\in\R_+^{1,n}$ and glue simplices using the diagonal matrix with diagonal $(1,\dots,1,-\sgn(\sigma_t^f))$.
\end{example}

We can verify that $\pi_1(L)\to\SO(n-k)$ is trivial from the cocycle. Order the vertices of $c$. An \emphasisText{embedding} of $c$ is a pair $(T_t, \sigma)$ with $\sigma\in S_n$ such that vertex $\sigma(0), \dots, \sigma(k)$ of $T_t$ correspond to the vertices of $c$ in that order (this aligns with \cite[\texttt{FaceEmbedding}]{Regina}). For each embedding $(T_t, \sigma)$ of $c$, take edges $(\sigma, k+1), \dots, (\sigma, n-1)$ of $\overline{T_t}$. The $\SO(1,n)$-matrices assigned to these edges fix $\H^k\subset\H^n$ point-wise and thus can be thought of as an $\SO(n-k)$-matrix. The edges form a connected complex. We need to verify that we get a trivial matrix for each loop in this complex. It is sufficient to check this for a generating set for loops. We can find such a generating set by picking a rooted spanning tree. That is, for each edge in the complement of the spanning tree, we obtain a loop by connecting the ends of the edge to the root through the spanning tree.

Note that the complex obtained for a subcell of $c$ contains the complex for the cell $c$, so some of these tests are redundant.

\begin{remark}
Given just a triangulation in higher dimensions, it is undecidable whether its geometric realization is a manifold. However, given the the geometric structure encoded in the vertex Gram matrices $G_t$, the above tests can prove that $\myTrig\setminus\myTrig_0$ is a topological manifold (which admits a hyperbolic structure). 
\end{remark}

\subsection{Developed fundamental polyhedron} \label{sec:higherFundPoly}

We are reusing the notation from Section~\ref{sec:trigCharts}. We can use Section~\ref{sec:cocycleDev} to obtain a developed fundamental polyhedron $P$ with vertices $\vec{v}^v_t$ from a triangulation $\myTrig$ with a geometric structure. Assume $P$ was developed such that simplex $T_0$ is in $\Id$-standard position. Then, outward-facing normal of face $n$ of $T_0$ is simply given by $(0,\dots,0,-1)\in\R^{1,n}$. We obtain the outward-facing normal $\vec{n}^f_t$ of any face $f$ of any simplex $T_t\subset P$ by applying the matrices from the cocycle on on $\overline{\myTrig}$.

To find the incenter $x\in\H^n$ of $T_0$, we solve the linear system of equations $x\cdot \vec{n}^f_t=-1$ and then rescale $x$. The rescaling factor also gives us the inradius $r=d(x, \vec{n}^f_t)$. We apply a transform to $P$ sending $x$ to $(1,0,\dots,0)$. This gives us the input $P$ to Algorithm~\ref{algorithm:epAlgorithm}.

\subsection{Cusp neighborhoods}

Let $\myTrig$ be a triangulation with an ideal geometric structure as described in Section~\ref{sec:hypStructHigher}. The vertex Gram matrices $G_t$ also determine horoballs and thus cusp neighborhoods $C_j$. The Euclidean edge lengths of the corresponding cusp cross sections can be computed by evaluating the expression for $d_{\partial B(\vec{v}^v_t)}(\vec{v}^{v'}_t,\vec{v}^{v''}_t)$ in Section~\ref{sec:hypDistances} using the entries in $G_t$. We can use the Cayley-Menger determinant to compute the volume of each Euclidean simplex in a cusp cross section from the edge lengths. Adding up these volumes, we obtain the volume of each $\partial C_j$.

We can now perform Step~\ref{step:scaleVert} of Algorithm~\ref{algorithm:epAlgorithm} to compute the Epstein-Penner decomposition. For Step~\ref{step:scaleToStd}, the scaling factor for one Euclidean simplex in the cusp cross section is given by $(\vec{v}^v_t\cdot \vec{n}^v_t)^{-1}$.

\subsection{Distances}

Step~\ref{step:tileEP} of Algorithm~\ref{algorithm:epAlgorithm} invokes Algorithm~\ref{algorithm:tilingAny} to tile about a point $K=x\in\H^n$. This requires computing $d(K, mT^f_t)=d(x,mT^f_t)=d(m^{-1}x,T^f_t)$. That is, we need to compute the distance $d(p, T^f_t)$ of a point $p$ to face $f$ of the $n$-simplex $T_t\subset P$. We more generally compute the distance $d(p, c)$ to any $k$-cell $c$ of $T_t$ with $k=n-1,\dots,1$ using the following recursive procedure.

With the same assumption as in Section~\ref{sec:higherFundPoly}, note that the interior of the cell $c$ spanned by $\vec{v}_0^0,\dots, \vec{v}_0^{k}$ is contained in $\H^k\cap \R^{1,k}_+\subset \H^n$ for $k=n,\dots,1$. Regarding $c$ as a simplex in $\H^k$, it is bounded to one side by $\H^{k-1}$ containing the adjacent cell $c'$ spanned by $\vec{v}_0^0,\dots, \vec{v}_0^{k-1}$ for $k=n,\dots,2$.

More generally, let $(c,c')$ be a pair of a $k$-cell and adjacent $k-1$-cell of simplex $T_t$ with $k=n,\dots,2$. We denote by $\alpha_t^{c,c'}$ a transform taking the interior of $c$ into $\H^k\cap \R^{1,k}$ and $c'$ into $\H^{k-1}$. We can again use the cocycle on $\overline{\myTrig}$ to compute $\alpha_t^{c,c'}$.

Similarly, we denote by $\alpha_t^c$ a transformation taking a $k$-cell $c$ of simplex $T_t$ with $k=n,\dots, 1$ into $\H^k$.

We use $(\_)_k$ for the $k$-th coordinate, that is $(x_0,\dots,x_n)_k=x_k$. We also use the projection
\[
\pi_k\from\H^n\to\H^n, (x_0,\dots, x_n)\mapsto (0,\dots,0, x_{k+1},\dots,x_n).
\]

Let $p\in\H^n$ and let $c$ be a $k$-cell of $T_t$ with $k=n-1,\dots,1$. If $k\geq 2$ and there is an adjacent $k-1$-cell $c'$ with $(\alpha_t^{c,c'} p)_k \leq 0$, then $d(p,c)=d(p,c')$. Otherwise,
\[
d(p,c)=\sinh^{-1} \left(\pi_k(\alpha_t^c p) \cdot \pi_k(\alpha_t^c p)\right).
\]

\begin{remark}
For finite simplices, the procedure needs to be extended by letting $k$ go down one more dimension.
\end{remark}

\clearpage

\appendix
\section{Notation}
\label{App:Notation}
For the convenience of the reader, we list some of the notation used in the paper.  

\newcolumntype{L}[1]{>{\raggedright\let\newline\\\arraybackslash\hspace{0pt}}p{#1}}
\newcolumntype{C}[1]{>{\centering\let\newline\\\arraybackslash\hspace{0pt}}p{#1}}
\newcolumntype{R}[1]{>{\raggedleft\let\newline\\\arraybackslash\hspace{0pt}}p{#1}}

\begin{longtable}{L{0.15\textwidth}L{0.79\textwidth}}%
\renewcommand{\arraystretch}{1.4}%
$\myTrig$ & 3-dimensional triangulation.\\
$T_t$ & Tetrahedron of $\myTrig$ --- abusing notation, also the embedding in $\H^3$.\\
$M$ & Hyperbolic manifold and/or Manifold that is the geometric realization of $\myTrig$.\\
$C_i$ & Cusp neighborhood about cusp $i$.\\
$A(C_i)$ & Area of $C_i$.\\
$A_M=(A_{ij})$ & Maximal cusp area matrix for $M$; see Definition~\ref{def:maxCuspAreaMatrix}.\\
$\Gamma$ & Subset of $\SO(1,3)$ such that $M\cong \Gamma\backslash\H^3$ (or $M_\text{filled}\cong\Gamma\backslash\H^3$ if $M$ is incomplete).\\
$m$ & Element in $\Gamma$.\\
$a=[\underline{a},\overline{a}]$ & Notation for interval; see Section~\ref{sec:intConv}.\\
$x\cdot x'$ & Inner product of signature $(-,+,+,+)$; see Section~\ref{sec:hyperboloidModel}.\\
$\H^3$ & Hyperboloid model of hyperbolic 3-space.\\
$\posLightCone$ & Future light-cone.\\
$B(l)$ & horoball about $l\in\posLightCone$.\\
$\widehat{x}$ or $(x)\postNorm$ & Projection onto $\H^3$.\\
$\hypPlane{n}$ & Plane with normal $n\in\R^{1,3}$.\\
$x_0x_1$ & Line with endpoints $x_0, x_1\in\posLightCone$.\\
$d(\_, \_)$ & Unsigned or signed distance in $\H^3$; see  Section~\ref{sec:hypDistances}.\\
$d_{x_0x_1}(x,x')$ & Signed distance when projected onto line $x_0x_1$ in $\H^3$.\\
$d_{\partial B(l)}(x, x')$ & Distance when projected onto horoball $B(l)$.\\
$T$ & Ideal triangle in $\H^3$.\\
$t, t'$ & Index of tetrahedron in $\myTrig$; see Section~\ref{sec:triangulations}.\\
$v, v'$ & Index of vertex in a tetrahedron.\\
$f, f'$ & Index of face in a tetrahedron.\\
$n^f_t$ & Index of tetrahedron neighboring $T_t$ across face $f$.\\
$\sigma^f_t$ & Face-gluing permutation for $T_t$ and face $f$.\\
$M_\text{filled}$ & If $M$ is incomplete, $M_\text{filled}$ is the complete manifold obtained by attaching circles. $M\to M_\text{filled}$ with equality if $M$ is complete; see Section~\ref{sec:geometricStruct}.\\
$z_t$ & Cross ratio/shape of $T_t$.\\
$\vec{v}_t^v$ & Vertices of tetrahedron $T_t$ in $\H^3\cup\posLightCone$; see Section~\ref{sec:trigCharts}.\\
$\vec{n}_t^f$ & Outward-facing normals for faces of $T_t$.\\
$T_t^f$ & Triangle making up face $f$ of $T_t$.\\
$g_t^f$ & Face-pairing matrix for $T_t$ and face $f$.\\
$P$ & Developed fundamental polyhedron; see Section~\ref{sec:devFundPoly}.\\
$\widetilde{V}$ & Vertices of universal cover $\widetilde{\myTrig}$; see Section~\ref{sec:cocycleDev}.\\
$\Psi$ & Vertex developing map $\widetilde{V}\to\overline{\H}^3$.\\
$mT_t$ & Lifted tetrahedron with $m\in\Gamma$; see Section~\ref{sec:algoInput}.\\
$K$ & Standard geometric object: a point $x\in\H^3$, line $x_0x_1$ with $x_0, x_1\in\posLightCone$ or horosphere $B(l)$ in $\posLightCone$; see Definition~\ref{def:standardObject}.\\
$B_r(K)$ & $r$-neighborhood of $K$, in $\H^3$, $T_t$ or $M$.\\
$\overline{B}_r(K)$ & Closed $r$-neighborhood of $K$, in $\H^3$, $T_t$ or $M$.\\
$(r_i,m_iT_{t_i})$ & (Object view) Stream of tiling radii and lifted tetrahedra; see Section~\ref{sec:tilePoint}. Output of tiling algorithm (Algorithm~\ref{algorithm:tilingAny}).\\
$(r_i, (m^{-1}_i K,T_{t_i}))$ & (Tetrahedra view) Stream of tiling radii and lifted objects; see Section~\ref{sec:tilingGeneralCase}.\\
$Q$ & Priority queue in Algorithm~\ref{algorithm:tilingAny}; see Section~\ref{sec:tilingAlgoDetailSec}.\\
$S$ & Set of lifted tetrahedra in $\Gamma_K\backslash\H^3$ in Algorithm~\ref{algorithm:tilingAny}; implemented in Section~\ref{sec:setLiftedTets}.\\
$D\from X\rightharpoonup Y$ & Dictionary; see Section~\ref{sec:dict}.\\
$\Eq_b$ & Equality predicate for discrete subset of $\H^3\cup\posLightCone$; see Section~\ref{sec:hyperbolicDict}.\\
$D_{\H^3\cup\posLightCone, b}$ & Dictionary with keys being in a discrete subset $X$ of $\H^3\cup\posLightCone$. $b$ is lower bound for $-x\cdot y$ with distinct $x,y\in X$. See Section~\ref{sec:hyperbolicDict}.\\
$p,r$ & Incenter and inradius of a tetrahedron $T_t$; see Section~\ref{sec:gammaDict}.\\
$d_M(K,K')$ & Distance between $K$ and $K'$ in $M$. Twice the embedding size of $K$ if $K$ and $K'$ are the same. See Section~\ref{sec:distStandard}.\\
$r\big[K\big]_i$ & Tiling radius for $K$ after $i$ steps of tiling algorithm.\\
$\Tiles\big[K\big]_{i,t}$ & Tiles for $K$ in tetrahedra view for tetrahedron $T_t$.\\
$d_{i,i'}(K,K')$ & Distance between tiles for $K$ and $K'$ after $i$ and $i'$ steps of the tiling algorithm.\\
$s_i$ & Cusp shape; see Section~\ref{sec:peripheralConventions}.\\
$\Delta^n$ & $n$-simplex. See Section~\ref{sec:natCocycleGram}.\\
$\overline{\Delta}^n$ & Permutahedron of order $n+1$ (which is a $n-1$-times truncated $n$-simplex).\\
$\sigma$ & Permutation of $\{0,\dots,n\}$ in $S_{n+1}$.\\
$(\sigma, k)$ & Edge label of $\overline{\Delta}^n$.\\
$\vec{v}^v$ & Vertex $v$ of simplex $\Delta^n$ in $\H^n\cup \mathbb{L}_{+}^n$.\\
$G$, $G*$ & Vertex Gram matrix and normal Gram matrix.\\
$\H^n$ & Hyperbolic $n$-space.\\
$\R^{1,n}$ & Minkowski space of signature $(-,+,\dots,+)$.\\
$\R^{1,n}_+, \R^{1,n}_-$ & Decomposition of $\R^{1,n}$.\\
$G_t$ & Vertex Gram matrix for simplex $T_t$; see Section~\ref{sec:hypStructHigher}.\\
$\overline{\myTrig}$ & Cell-complex obtained by replacing each $n$-simplex of $\myTrig$ by $\overline{\Delta}^n$; also see \cite[Section~3]{matthiasVerifyingFinite}.\\
\end{longtable}

\bibliographystyle{hamsalphaMatthias}
\renewcommand{\path}[1]{#1}

\bibliography{embeddingSizes}

@article {matthiasVerifyingFinite,
    AUTHOR = {Goerner, Matthias},
     TITLE = {Verified computations for closed hyperbolic 3-manifolds},
   JOURNAL = {Bull. Lond. Math. Soc.},
  FJOURNAL = {Bulletin of the London Mathematical Society},
    VOLUME = {53},
      YEAR = {2021},
    NUMBER = {2},
     PAGES = {596--618},
      ISSN = {0024-6093},
   MRCLASS = {57K32 (57Q15 65G20)},
  MRNUMBER = {4239199},
       DOI = {10.1112/blms.12445},
      eprint = {1904.12095},
    archivePrefix = { arXiv }
}

@Article{hikmot,
  author   = {Hoffman, N. and Ichihara, K. and Kashiwagi, M. and Masai, H. and Oishi, S. and Takayasu, A.},
  title    = {Verified computations for hyperbolic 3-manifolds},
  journal  = {Experiment. Math.},
  year     = {2016},
  volume   = {25},
  number   = {1},
  pages    = {66--78},
  note     = {\url{http://www.oishi.info.waseda.ac.jp/~takayasu/hikmot/}},
  eprint   = {1310.3410},
    archivePrefix = { arXiv },
  fjournal = {Experimental Mathematics},
  mrnumber = {3424833},
  zbl      = {1337.57044},
}

@misc{SnapPy,
     author={Culler, Marc and Dunfield, Nathan M. and Goerner,
     Matthias and Weeks, Jeffrey R.},
     title={Snap{P}y, a computer program for studying the geometry and topology of $3$-manifolds},
     howpublished={available at \url{http://snappy.computop.org/} ({V}ersion 3.2)},
     year={2025},
}

@misc{sagemath,
  Key          = {SageMath},
  Author       = {The {Sage Developers}},
  Title        = {{S}ageMath, the {S}age {M}athematics {S}oftware {S}ystem ({V}ersion 10.5)},
  note         = {\url{https://www.sagemath.org}},
  Year         = {2025},
}

@phdthesis {heardThesis,
    AUTHOR = {Heard, Damian},
     TITLE = {Computation of hyperbolic structures on 3-dimensional orbifolds},
      YEAR = {2005},
    SCHOOL = {University of Melbourne},
      NOTE = {\url{https://github.com/DamianHeard/orb-thesis/}}
}

@misc {casson:geo,
    author = {Casson, Andrew},
     TITLE = {Geo, a program for geometrizing 3-manifolds},
      NOTE = {\url{http://computop.org/}}
}

@Article{hwcensus,
    Author = {Craig D. {Hodgson} and Jeffrey R. {Weeks}},
    Title = {{Symmetries, isometries and length spectra of closed hyperbolic three-manifolds}},
    FJournal = {{Experimental Mathematics}},
    Journal = {{Experiment. Math.}},
    ISSN = {1058-6458; 1944-950X/e},
    Volume = {3},
    Number = {4},
    Pages = {261--274},
    Year = {1994},
    Publisher = {Taylor \& Francis, Philadelphia, PA},
    DOI = {10.1080/10586458.1994.10504296},
    MSC2010 = {57M50 57N10 57-04},
    Zbl = {0841.57020}
}

@misc{ThurstonNotes,
AUTHOR = {Thurston, William P.},
TITLE = {The Geometry and Topology of Three-Manifolds},
YEAR = {1980},
NOTE = {Princeton lecture notes, available at \url{https://library.slmath.org/nonmsri/gt3m/}}
 }

@Article{moser,
    Author = {Harriet {Moser}},
    Title = {{Proving a manifold to be hyperbolic once it has been approximated to be so}},
    FJournal = {{Algebraic \& Geometric Topology}},
    Journal = {{Algebr. Geom. Topol.}},
    ISSN = {1472-2747; 1472-2739/e},
    Volume = {9},
    Number = {1},
    Pages = {103--133},
    Year = {2009},
    Publisher = {Mathematical Sciences Publishers (MSP), Berkeley, CA; Geometry \& Topology Publications c/o University of Warwick, Mathematics Institute, Coventry},
    Language = {English},
    DOI = {10.2140/agt.2009.9.103},
    MSC2010 = {57M50 57N16 51H20 54E50 57-04 53A35},
    Zbl = {1170.57015},
    eprint = {0809.1203},
    archivePrefix = { arXiv }
}

@book {ratcliffe:hyp,
    AUTHOR = {Ratcliffe, John G.},
     TITLE = {Foundations of hyperbolic manifolds},
    SERIES = {Graduate Texts in Mathematics},
    VOLUME = {149},
 PUBLISHER = {Springer-Verlag},
   ADDRESS = {New York},
      YEAR = {1994},
     PAGES = {xii+747},
      ISBN = {0-387-94348-X},
   MRCLASS = {57M50 (20H10 30F40 51M10)},
  MRNUMBER = {1299730 (95j:57011)},
MRREVIEWER = {Colin C. Adams},
}

@article{PetronioWeeks:partiallyFlatTrig,
    author = "Petronio, Carlo and Weeks, Jeffrey R.",
  fjournal = "Osaka Journal of Mathematics",
   journal = "Osaka J. Math.",
    number = "2",
     pages = "453--466",
 publisher = "Osaka University and Osaka City University, Departments of Mathematics",
     title = "Partially flat ideal triangulations of cusped hyperbolic 3-manifolds",
       url = "http://projecteuclid.org/euclid.ojm/1200789209",
    volume = "37",
      year = "2000",
     Zbl = {0952.57003},
   MRNUMBER = {1772844 (2001g:57032)}
}

@misc{Regina,
    author = {Benjamin A. Burton and Ryan Budney and William Pettersson and others},
    title = {Regina: Software for low-dimensional topology},
    howpublished = {available at \url{https://regina-normal.github.io/} ({V}ersion 7.3)},
    year = {2023}
}

@article{cohomologyFractals,
    author = {Bachmann, David and Goerner, Matthias and Schleimer, Saul and Segerman, Henry},
    title = {Cohomology fractals, {C}annon-{T}hurston maps, and the geodesic flow},
    journal = {Experiment. Math.},
    fjournal = {Experimental Mathematics},
    volume = {31},
    year={2022},
    number={4},
    pages={1047--1085},
    doi = {10.1080/10586458.2021.1994059},
    eprint = {2010.05840v2},
    archivePrefix = { arXiv }
}

@article { GGZ:GluingEquations,
         author = { Garoufalidis, Stavros and Goerner, Matthias and Zickert, Christian K. },
          title = { Gluing equations for {$\rm{PGL}(n,\mathbb C)$}-representations of 3-manifolds },
        journal = { Algebr. Geom. Topol. },
       fJournal = { Algebraic \& Geometric Topology },
         volume = { 15 },
           year = { 2015 },
         number = { 1 },
          pages = { 565--622 },
         eprint = { 1207.6711 },
  archivePrefix = { arXiv },
            Zbl = { 1347.57014 }
}

@article { GGZ:PtolemyField,
         author = { Garoufalidis, Stavros and Goerner, Matthias and Zickert, Christian K. },
          title = { The {P}tolemy field of 3-manifold representations },
        journal = { Algebr. Geom. Topol. },
       fJournal = { Algebraic \& Geometric Topology },
         volume = { 15 },
           year = { 2015 },
         number = { 1 },
          pages = { 371--397 },
         eprint = { 1401.5542 }, 
  archivePrefix = { arXiv },
            Zbl = { 1322.57018 },
}

@Article{marche:CS,
 Author = {Julien {March\'e}},
 Title = {{Geometric interpretation of simplicial formulas for the Chern-Simons invariant}},
 FJournal = {{Algebraic \& Geometric Topology}},
 Journal = {{Algebr. Geom. Topol.}},
 ISSN = {1472-2747; 1472-2739/e},
 Volume = {12},
 Number = {2},
 Pages = {805--827},
 Year = {2012},
 Publisher = {Mathematical Sciences Publishers (MSP), Berkeley, CA; Geometry \& Topology Publications c/o University of Warwick, Mathematics Institute, Coventry},
 Language = {English},
 MSC2010 = {57M27 58J28},
 Zbl = {1251.57015},
  eprint = { 1011.3139},
    archivePrefix = { arXiv }
}

@article{GTZ:ptolemyCoords,
    author = {Stavros {Garoufalidis} and Dylan Thurston and Christian {Zickert}},
     title = {The complex volume of $\mathrm{SL}(n,\mathbb{C})$-representations of 3-manifolds},
    fjournal = {Duke Mathematical Journal},
     journal = {Duke Math. J.},
     year = {2015},
     number = {11},
     volume={164},
     pages = {2099--2160},
     eprint = {1111.2828},
     archivePrefix = "arXiv"
}

@article{zickert:volume,
  author = {Christian Zickert},
   title = {The volume and Chern-Simons invariant of a representation},
    fjournal = {Duke Mathematical Journal},
     journal = {Duke Math. J.},
     year = {2009},
     number = {3},
     volume={150},
     pages={489--532},
     eprint={0710.2049},
     archivePrefix = "arXiv"
}

@article {Agol:Six,
    AUTHOR = {Agol, Ian},
     TITLE = {Bounds on exceptional {D}ehn filling},
   JOURNAL = {Geom. Topol.},
  FJOURNAL = {Geometry and Topology},
    VOLUME = {4},
      YEAR = {2000},
     PAGES = {431--449},
      ISSN = {1465-3060},
   MRCLASS = {57M50 (57M25 57M27 57S25)},
  MRNUMBER = {1799796 (2001j:57019)},
MRREVIEWER = {Danny C. Calegari},
       DOI = {10.2140/gt.2000.4.431},
       URL = {http://dx.doi.org/10.2140/gt.2000.4.431},
   EPRINT = {math/9906183}
}

@article {Lackenby:Six,
    AUTHOR = {Lackenby, Marc},
     TITLE = {Word hyperbolic {D}ehn surgery},
   JOURNAL = {Invent. Math.},
  FJOURNAL = {Inventiones Mathematicae},
    VOLUME = {140},
      YEAR = {2000},
    NUMBER = {2},
     PAGES = {243--282},
      ISSN = {0020-9910},
     CODEN = {INVMBH},
   MRCLASS = {57M07 (20F65 20F67 57M05 57N10)},
  MRNUMBER = {1756996 (2001m:57003)},
MRREVIEWER = {William H. Jaco},
       DOI = {10.1007/s002220000047},
       URL = {http://dx.doi.org/10.1007/s002220000047},
    EPRINT = {math/9808120},
 archivePrefix={arXiv}
}

@Article{fiveChainLinkSlopes,
 Author = {Bruno {Martelli} and Carlo {Petronio} and Fionntan {Roukema}},
 Title = {{Exceptional Dehn surgery on the minimally twisted five-chain link}},
 FJournal = {{Communications in Analysis and Geometry}},
 Journal = {{Commun. Anal. Geom.}},
 ISSN = {1019-8385; 1944-9992/e},
 Volume = {22},
 Number = {4},
 Pages = {689--735},
 Year = {2014},
 Publisher = {International Press of Boston, Somerville, MA},
 Language = {English},
 MSC2010 = {57M25 57M50},
 Zbl = {1307.57009},
 eprint={1109.0903},
 archivePrefix={arXiv}
}

@Article{altKnotSlopes,
 Author = {Kazuhiro {Ichihara} and Hidetoshi {Masai}},
 Title = {{Exceptional surgeries on alternating knots}},
 FJournal = {{Communications in Analysis and Geometry}},
 Journal = {{Commun. Anal. Geom.}},
 ISSN = {1019-8385; 1944-9992/e},
 Volume = {24},
 Number = {2},
 Pages = {337--377},
 Year = {2016},
 Publisher = {International Press of Boston, Somerville, MA},
 Language = {English},
 MSC2010 = {57M50 57M25},
 Zbl = {1350.57020},
 eprint={1310.3472},
 archivePrefix={arXiv}
}

@article{threeChainLinkSlopes,
     Author = {Bruno {Martelli} and Carlo {Petronio}},
 Title = {{Dehn filling of the ``magic'' 3-manifold}},
 FJournal = {{Communications in Analysis and Geometry}},
 Journal = {{Commun. Anal. Geom.}},
 ISSN = {1019-8385; 1944-9992/e},
 Volume = {14},
 Number = {5},
 Pages = {969--1026},
 Year = {2006},
 Publisher = {International Press of Boston, Somerville, MA},
 Language = {English},
 MSC2010 = {57M50 57M25 57N10 57R15},
 eprint={math/0204228},
 archivePrefix={arXiv},
 Zbl = {1118.57018}
}

@InCollection{excepSlopeCensus,
 Author = {Nathan M. {Dunfield}},
 Title = {{A census of exceptional Dehn fillings}},
 BookTitle = {{Characters in low-dimensional topology. A conference celebrating the work of Steven Boyer, Universit\'e du Qu\'ebec \`a Montr\'eal, Montr\'eal, Qu\'ebec, Canada, June 2--6, 2018}},
 ISBN = {978-1-4704-5209-4/pbk; 978-1-4704-6135-5/ebook},
 Pages = {143--155},
 Year = {2020},
 Publisher = {Providence, RI: American Mathematical Society (AMS); Montreal: Centre de Recherches Math\'ematiques (CRM)},
 Language = {English},
 MSC2010 = {57K35 57K32},
  eprint = {1812.11940},
 archivePrefix={arXiv},
    Zbl = {07316044}
}

@article{sevenChainLinkSlopes,
 Author = {Bruno {Martelli}},
 Title = {{D}ehn surgery on the minimally twisted seven-chain link},
 Year = {2021},
 fjournal = {Communications in Analysis and Geometry},
 journal = {Commun. Anal. Geom.},
 volume = {29},
 number = {7},
 pages = {1597--1641},
 eprint={1808.08430},
 archivePrefix={arXiv}
}

@article {perel1,
  author = {Grisha Perelman},
  title = {The entropy formula for the {R}icci flow and its geometric applications},
  year = {2002},
  eprint={math.DG/0211159},
  archivePrefix={arXiv}
}

@article {perel2,
  author = {Grisha Perelman},
  title = {{R}icci flow with surgery on three-manifolds},
  year = {2003},
  eprint ={math.DG/0303109},
  archivePrefix={arXiv}
}

@article {EP,
    AUTHOR = {Epstein, D. B. A. and Penner, R. C.},
     TITLE = {Euclidean decompositions of noncompact hyperbolic manifolds},
   JOURNAL = {J. Differential Geom.},
  FJOURNAL = {Journal of Differential Geometry},
    VOLUME = {27},
      YEAR = {1988},
    NUMBER = {1},
     PAGES = {67--80},
      ISSN = {0022-040X},
     CODEN = {JDGEAS},
   MRCLASS = {57N15 (20H10 22E40 51M10)},
  MRNUMBER = {918457 (89a:57020)},
MRREVIEWER = {N. V. Ivanov},
       URL = {http://projecteuclid.org/euclid.jdg/1214441650},
    Zbl = {0611.53036}
}

@book {PurcellKnotTheory,
      title={Hyperbolic knot theory ({G}raduate {S}tudies in {M}athematics)}, 
      author={Jessica S. Purcell},
      year={2020},
      PUBLISHER = {American Mathematical Society, Providence, RI},
      eprint={2002.12652},
    archivePrefix = { arXiv }
}

@book {BenedettiPetronio,
    AUTHOR = {Benedetti, Riccardo and Petronio, Carlo},
     TITLE = {Lectures on hyperbolic geometry},
    SERIES = {Universitext},
 PUBLISHER = {Springer-Verlag},
   ADDRESS = {Berlin},
      YEAR = {1992},
     PAGES = {xiv+330},
      ISBN = {3-540-55534-X},
   MRCLASS = {57M50 (30F40 30F60 51M10 57N10)},
  MRNUMBER = {MR1219310 (94e:57015)},
MRREVIEWER = {Colin C. Adams},
}

@InCollection{burton:encode,
    Author = {Benjamin A. {Burton}},
    Title = {{The Pachner graph and the simplification of 3-sphere triangulations.}},
    BookTitle = {{Proceedings of the 27th annual symposium on computational geometry, SoCG 2011, Paris, France, June 13--15, 2011}},
    ISBN = {978-1-4503-0682-9},
    Pages = {153--162},
    Year = {2011},
    Publisher = {New York, NY: Association for Computing Machinery (ACM)},
    Language = {English},
    DOI = {10.1145/1998196.1998220},
    MSC2010 = {05C10 05C85 68U05 57N10 32B25},
    Zbl = {1283.05065},
   MRCLASS = {68U05 (05C38 05C85 57Q15)},
  MRNUMBER = {2919606},
  EPRINT   = {1110.6080},
    archivePrefix = { arXiv }
}

@article{heapsort,
   Author = { Williams, John William Joseph},
   year = { 1964},
   title = {Algorithm 232: {H}eapsort},
   journal = {Communications of the ACM},
   volume = {7},
   number = {6},
   pages = {347--348}
}

@INPROCEEDINGS{redBlack,
  author={Guibas, Leo J. and Sedgewick, Robert},
  booktitle={19th Annual Symposium on Foundations of Computer Science (sfcs 1978)}, 
  title={A dichromatic framework for balanced trees}, 
  year={1978},
  volume={},
  number={},
  pages={8--21},
  doi={10.1109/SFCS.1978.3}}

@article{ghht:lenSpec,
         Author = {Goerner, Matthias and Haraway, Robert C. III and Hoffman, Neil R. and Trnkova, Maria},
         Title = {Verified Length Spectrum and {M}argulis number for Hyperbolic 3-Manifolds},
         Year = {2026},
         note = {in preparation}}

@article{goerner:drilling,
    Author = {Goerner, Matthias},
    Title = {Drilling Geodesics in Hyperbolic 3-Manifolds and the Closed Isometry Signature},
    Year = {2026},
    note = {in preparation}}

@Article{weeks:canonical,
 Author = {Weeks, Jeffrey R.},
 Title = {Convex hulls and isometries of cusped hyperbolic 3-manifolds},
 Journal = {Topology and its Applications},
 ISSN = {0166-8641},
 Volume = {52},
 Number = {2},
 Pages = {127--149},
 Year = {1993},
 DOI = {10.1016/0166-8641(93)90032-9},
 zbMATH = {446795},
 Zbl = {0808.57005}
}

@article{sakumaWeeks:tilt,
 Author = {Sakuma, Makoto and Weeks, Jeffrey R.},
 Title = {The generalized tilt formula},
 FJournal = {Geometriae Dedicata},
 Journal = {Geom. Dedicata},
 ISSN = {0046-5755},
 Volume = {55},
 Number = {2},
 Pages = {115--123},
 Year = {1995},
 DOI = {10.1007/BF01264924},
 zbMATH = {806223},
 Zbl = {0834.57009}
}

@article{adams:cuspDensity,
Author = {Adams, Colin and Kaplan-Kelly, Rose and Moore, Michael and Shapiro, Brandon and Sridhar, Shruthi and Wakefield, Joshua},
 Title = {Densities of hyperbolic cusp invariants of knots and links},
 FJournal = {Proceedings of the American Mathematical Society},
 Journal = {Proc. Am. Math. Soc.},
 ISSN = {0002-9939},
 Volume = {146},
 Number = {9},
 Pages = {4073--4089},
 Year = {2018},
 Language = {English},
 DOI = {10.1090/proc/14068},
 zbMATH = {6904501},
 Zbl = {1397.57033},
 eprint={1701.03479}
}

@InCollection{Ushijima:volFinite,
 Author = {Ushijima, Akira},
 Title = {A volume formula for generalised hyperbolic tetrahedra},
 BookTitle = {Non-Euclidean geometries. J\'anos Bolyai memorial volume. Papers from the international conference on hyperbolic geometry, Budapest, Hungary, July 6--12, 2002},
 ISBN = {0-387-29554-2; 0-387-29555-0},
 Pages = {249--265},
 Year = {2006},
 Publisher = {New York, NY: Springer},
 Language = {English},
 zbMATH = {5046181},
 Zbl = {1096.52006},
 eprint={math/0309216}
}

@Article{quickhull,
 Author = {Barber, C. Bradford and Dobkin, David P. and Huhdanpaa, Hannu},
 Title = {The quickhull algorithm for convex hulls},
 FJournal = {ACM Transactions on Mathematical Software},
 Journal = {ACM Trans. Math. Softw.},
 ISSN = {0098-3500},
 Volume = {22},
 Number = {4},
 Pages = {469--483},
 Year = {1996},
 Language = {English},
 DOI = {10.1145/235815.235821},
 URL = {www.acm.org/pubs/contents/journals/toms/1996-22/},
 zbMATH = {1101869},
 Zbl = {0884.65145}
}

@misc{FMP:ographs,
    author = {Frigerio, Roberto and Martelli, Bruno and Pertronio, Carlo},
     title = {Ographs, a computer program for computing structures on hyperbolic 3-manifolds with geodesic boundary},
      note = {\url{https://people.dm.unipi.it/petronio/progs.html}}
}

@article { FGGTV:TetrahedralCensus,
         author = { Fominykh, Evgeny and Garoufalidis, Stavros and Goerner, Matthias and Tarkaev, Vladimir and Vesnin, Andrei },
          title = { A census of tetrahedral hyperbolic manifolds },
        journal = { Experiment. Math. },
       fJournal = { Experimental Mathematics },
         volume = { 25 },
           year = { 2016 },
         number = { 4 },
          pages = { 466--481 },
         eprint = { 1502.00383 },
  archivePrefix = { arXiv },
            Zbl = { 1344.57009 },
       MRNUMBER = { 3499710 },
}

@article{fps:margulis,
 Author = { Futer, David and Purcell, Jessica S. and Schleimer, Saul },
 Title = {Effective distance between nested {M}argulis tubes},
 Journal = {Trans. Amer. Math. Soc.},
 FJournal = {Transactions of the American Mathematical Society},
 Volume ={372},
 Number = {6},
 Year = {2019},
 Pages = {4211--4237},
 eprint = {1801.05342}
}

@article{meyerhoff:orthoSpec,
 author = {Meyerhoff, G. Robert},
 title = {The ortho-length spectrum for hyperbolic 3-manifolds},
 fjournal = {The Quarterly Journal of Mathematics. Oxford Second Series},
 journal = {Q. J. Math., Oxf. II. Ser.},
 issn = {0033-5606},
 volume = {47},
 number = {187},
 pages = {349--359},
 year = {1996},
 language = {English},
 zbMATH = {938703},
 Zbl = {0864.57010}
}

\end{document}